\documentclass[aap,MSNbibl,secthm,seceqn,citesort,dvips]{arximspdf}
\usepackage{mathbh}

\doi{10.1214/10-AAP737}
\volume{21}
\issue{5}
\pubyear{2011}
\firstpage{1795}
\lastpage{1826}

\makeatletter
\def\bsuffix #1{#1}
\newcommand{\fracc}[2]{{(#1/#2)}}
\renewcommand{\epsilon}{\varepsilon}
\newcommand{\eqdef}{\stackrel{\mathrm{def}}{=}}
\newcommand{\straT}{\mathcal{C}}
\newcommand{\straTopt}{{\mathcal{C}^\star}}
\newcommand{\mcH}{\mathcal{H}}
\newcommand{\mcZ}{\mathcal{Z}}

\renewcommand{\Pr}{\mathbb{P}}
\newcommand{\E}{\mathbb{E}}
\newcommand{\Generator}{\mathcal{G}}

\newcommand{\R}{{\mathbb R}}
\newcommand{\N}{{\mathbb N}}
\newcommand{\Filt}{{\mathcal F}}
\newcommand{\FiltMP}{{\mathcal F}}
\newcommand{\FiltH}{\mathcal{H}}
\renewcommand{\SS}{\mathcal{S}}
\newcommand{\SSd}{D}
\newcommand{\exitST}{\mathfrak{m}}
\newcommand{\tauopt}{\tau^\straTopt}

\newcommand{\V}{V}
\newcommand{\eqref}[1]{(\ref{#1})}

\newcommand{\Indi}[1]{\mathbh{1}_{[#1]}}
\newtheorem{theorem}{Theorem}[section]
\newtheorem{lemma}[theorem]{Lemma}
\newproclaim{remark}[theorem]{Remark}
\newproclaim{definition}[theorem]{Definition}
\makeatother

\begin{document}
\begin{frontmatter}

\title{Minimizing the time to a decision}
\runtitle{Minimizing the decision time}

\begin{aug}
\author[A]{\fnms{Saul} \snm{Jacka}\ead[label=e1]{S.D.Jacka@warwick.ac.uk}},
\author[A]{\fnms{Jon} \snm{Warren}\ead[label=e2]{J.Warren@warwick.ac.uk}}
\and
\author[A]{\fnms{Peter} \snm{Windridge}\corref{}\ead[label=e3]{Peter.Windridge@warwick.ac.uk}}
\runauthor{S. Jacka, J. Warren and P. Windridge}
\affiliation{University of Warwick}
\address[A]{Department of Statistics\\
University of Warwick\\
Coventry CV4 7AL\\
United Kingdom\\
\printead{e1}\\
\hphantom{\textsc{E-mail}:\ }\printead*{e2}\\
\hphantom{\textsc{E-mail}:\ }\printead*{e3}} %adresu isvedimo komanda gale!
\end{aug}

% HISTORY:
\received{\smonth{5} \syear{2010}}
\revised{\smonth{7} \syear{2010}}

% ABSTRACT
%
\begin{abstract}
Suppose we have three independent copies of a regular diffusion on $[0,1]$
with absorbing boundaries. Of these diffusions, either at least two are
absorbed at the upper boundary or at least two at the lower boundary.
In this way, they determine a majority decision between
0 and 1. We show that the strategy that always runs the diffusion whose
value is currently between the other two reveals the majority decision
whilst minimizing the total time spent running the processes.
\end{abstract}

% KEYWORDS
%

\end{frontmatter}

%s1 ###
\section{Introduction}\label{sec1}

Let $X_1,X_2$ and $X_3$ be three independent copies of a
regular diffusion on $[0,1]$ with absorbing boundaries.
Eventually, either at least two of the diffusions are absorbed
at the upper boundary of the interval or at least two are absorbed
at the lower boundary. In this way, the diffusions determine a
\textit{majority decision} between 0 and 1.

In order to identify this decision, we run the three processes---not
simultaneously, but switching from one to another---until we
observe at least two of them reaching a common boundary point. Our aim is
to switch between the processes in a way that minimizes the total time
required to find the majority decision.

More precisely, we allocate our time between the three
processes according to a suitably adapted $[0,\infty)^3$-valued increasing
process $\straT$ with $\sum_{i=1}^3 \straT_i(t) = t$. Such a process
is called a \textit{strategy} and $\straT_i(t)$ represents
the amount of time spent observing $X_i$ after $t \geq0$ units
of calendar time have elapsed. Accordingly, the process we observe is
\[
X^\straT\eqdef\bigl( X_1(\straT_1(t)), X_2(\straT_2(t)), X_3(\straT
_3(t)); t \geq0\bigr),
\]
and the \textit{decision time} $\tau^\straT$ for the strategy $\straT
$ is the first
time that two components of $X^\straT$ are absorbed at the same end point
of $[0,1]$, that is,
\[
\tau^\straT\eqdef\inf\bigl\{t \geq0\dvtx  X^\straT_i(t) = X^\straT_j(t)
\in\{0,1\} \mbox{ for distinct } i,j\bigr\}.
\]

In this paper, we find a strategy $\straTopt$ that minimizes this
time. Roughly speaking, $\straTopt$ runs whichever diffusion is
currently observed to have
``middle value'' (see Lemma \ref{l:stratexists} for a precise description).
Our main theorem is that the decision time $\tau^{\straTopt}$ of this
strategy is the \textit{stochastic minimum} of all possible decision
times, that is, the following theorem holds.

\begin{theorem}\label{t:stochminimality}
The decision time $\tau^{\straTopt}$ of the ``run the middle''
strategy $\straTopt$ given in Lemma \ref{l:stratexists} satisfies
\[
\Pr(\tau^{\straTopt} > t) = \inf_{\straT} \Pr(\tau^\straT> t)
  \qquad \mbox{for every } t \geq0,
\]
where the infimum is taken over all strategies and $\tau^\straT$
is the corresponding decision time.
\end{theorem}

The result fits with the existing literature on optimal dynamic
resource allocation
(see Section~\ref{ss:banditsetc} below) and we find it interesting in its own right.
However, our original motivation for introducing the model came from
the so-called
``recursive ternary majority'' problem, which can be described as follows.
Take the complete ternary tree on $n$ levels, place independent Bernoulli($1/2$)
variables on each of the $3^n$ leaves and define internal
nodes to take the majority value of their three children.
We must find the value of the root node by sequentially revealing
leaves, one after the other, paying \textsterling1 for each leaf revealed.
The quantity of concern is the expected cost $r_n$ of
the optimal strategy. Surprisingly, this number is not known for $n > 3$
and there seems little prospect of finding it. Interest has rather
focused on the asymptotic behavior of $r_n$, as this has more relevance in
complexity theory. In particular, the limit
\[
\gamma\eqdef\lim_{n\to\infty} r_n^{1/n},
\]
which exists by a sub-additivity argument, has attracted the attention
of several researchers recently. The best nontrivial bounds
are $9/4 \leq\gamma\leq2.471$ (the lower bound follows from
arguments in Section~{3} of \cite{peres2007random},
the upper bound from numerics).

Our idea was to find a better lower bound for
$\gamma$ by considering a continuous approximation
to the large $n$ tree. It was this continuous approximation that
inspired the diffusive model
introduced in this paper. However, we caution that the results we present
here do not shed light on the value of $\gamma$.% -- our approach has
%so far failed in this regard.

%s1.1 ###
\subsection{Dynamic resource allocation}\label{ss:banditsetc}
Our problem concerns optimal dynamic resource allocation
in continuous time. The most widely studied example of this is
the continuous multi-armed bandit problem (see, e.g., El Karoui
and Karatzas \cite{elkaroui1994dap}, Mandelbaum and Kaspi \cite
{kaspi98multi-armed}). Here, a gambler chooses
the rates at which he will pull the arms on different slot machines.
Each slot machine rewards the gambler at rates which follow a stochastic
process independent of the reward processes for the other machines.
These general bandit problems find application in several fields
where agents must choose between exploration and exploitation,
typified in economics % (see the short survey \cite{bergemann-bandit})
and clinical trials.
An optimal strategy is easy to describe. Associated to each machine
is a process known as the Gittins index, which may be interpreted as
the equitable surrender value. It is a celebrated theorem that at each instant,
we should play whichever machine currently has the largest Gittins
index. This
is in direct analogy to the discrete time result of Gittins and Jones
\cite{Git74}.

There is no optimal strategy of index type for our problem.
This reflects the fact that the reward processes
associated to running each of the diffusions are not independent---once two
of the diffusions are absorbed, it may be pointless to run the third.

In \cite{mandelbaum1990osb}, a different dynamic allocation problem is
considered. It has a similar
flavor in that one must choose the rates at which to run
%one chooses the rates at which to run
two Brownian motions on $[0,1]$, and we stop once \textit{one} of the
processes hits an endpoint.
The rates are chosen to maximize a terminal payoff, as specified by a function
defined on the boundary of the square (the generalization of this
problem to several
Brownian motions is considered in \cite{vanderbei92}). An optimal
strategy is determined
by a partition of the square into regions of indifference, preference
for the
first Brownian motion and preference for the second. However, there is
no notion of a reward (cost) being accrued as in our problem.

So, our problem, in which
time is costly and there is a terminal cost of infinity
for stopping on a part of $\partial\SS$ which does not determine a
majority decision,
could be seen as lying between continuous bandits and the Brownian
switching in \cite{mandelbaum1990osb}.
Furthermore, although we adopt the framework of the aforementioned
problems, our proof has a different mathematical anatomy.

%s1.2 ###
\subsection{Overview of paper}
The rest of the paper is laid out as follows. Section \ref{s:probstatement}
contains a precise statement of the problem and our assumptions and a
clarification of
Theorem \ref{t:stochminimality}. The proof of this theorem begins in
Section~\ref{sec2},
where we show that the Laplace transform of the distribution
of the decision time $\tauopt$
solves certain differential equations. This fact is then
used in Section~\ref{sec3} to show that the tail of $\tau^{\straTopt}$
solves, in a certain sense, the appropriate Hamilton--Jacobi--Bellman equation.
From here, martingale optimality arguments complete the proof.
Section~\ref{sec4} shows the existence and uniqueness of the strategy $\straTopt$
and in Section~\ref{sec5} we explain the connection between the controlled process
and doubly perturbed diffusions. In the final section, we make a conjecture
about an extension to the model.
% and then, to close, we ask a few questions
% relating to the discrete recursive majority of three problem that
% motivated us originally.

%s1.3 ###
\subsection{Problem statement and solution}\label{s:probstatement}

We are given a complete probability space $(\Omega, \Filt, \Pr)$
supporting three
independent It\^{o} diffusions $(X_i(t),  t \geq0)$, $i \in\V= \{
1,2,3\}$,
each of which is started in the unit interval $[0,1]$ and absorbed
at the endpoints. The diffusions all satisfy the same stochastic
differential equation
%
%e1.1 ###
\begin{equation}\label{e:itosde}
dX_i(t) = \sigma(X_i(t))\,dB_i(t) + \mu(X_i(t))\,dt,  \qquad   t \geq0,
\end{equation}
where $\sigma\dvtx [0,1]\to(0,\infty)$ is continuous, $\mu\dvtx [0,1]\to\R$
is Borel and
$(B_i(t),  t \geq0)$, $i \in\V$, are independent Brownian motions.

We denote by $\SS$ the unit cube $[0,1]^3$, by $\R_+$ the set of
nonnegative real numbers $[0,\infty)$ and $\preceq$ its usual partial
order on $\R^3_+$.
It is assumed that we have a standard Markovian setup, that is, there
is a
family of probability measures $(\Pr_x,   x \in\SS)$ under which
$X(0) = x$ almost surely and the filtration $\Filt_i =  ( \Filt
_i(t),  t \geq0 )$
generated by $X_i$ is augmented to satisfy the usual conditions.

From here, we adopt the framework for continuous dynamic allocation
models proposed by Mandelbaum in \cite{mandelbaum1987cma}. This approach
relies on the theory of multiparameter time changes; the reader may consult
\hyperref[s:intromultiparam]{Appendix}   for a short summary of this.

For $\eta\in\R_+^3$, we define the $\sigma$-algebra
\[
\FiltMP(\eta) \eqdef\sigma(\Filt_1(\eta_1), \Filt_2(\eta_2),
\Filt_3(\eta_3)),
\]
which corresponds to the information revealed by running
$X_i$ for $\eta_i$ units of time. The family $(\FiltMP(\eta),  \eta
\in\R_+^3)$ is called
a \textit{multiparameter filtration} and satisfies the ``usual conditions''
of right continuity, completeness and property (F4) of Cairoli and
Walsh \cite{cairoli75}.
It is in terms of this filtration that we define the sense
in which our strategies must be adapted.

A \textit{strategy} is an $\R_+^3$-valued stochastic process
\[
\straT=  \bigl(\straT_1(t), \straT_2(t), \straT_3(t);  t \geq
0 \bigr)
\]
such that:
\begin{longlist}[(C3)]
\item[(C1)] for $i = 1,2,3$, $\straT_i(0) = 0$ and $\straT_i(\cdot
)$ is nondecreasing,
\item[(C2)] for every $t \geq0$, $\straT_1(t) + \straT_2(t) +
\straT_3(t) = t$ and
\item[(C3)] $\straT(t)$ is a stopping ``point'' of the multiparameter
filtration $(\FiltMP(\eta),  \eta\in\R_+^3)$, that is,
\[
\{\straT(t) \preceq\eta\} \in\FiltMP(\eta)  \qquad  \mbox
{for every }  \eta\in\R_+^3.
\]
\end{longlist}

\begin{remark}
In the language of multiparameter processes, $\straT$ is an
\textit{optional increasing path} after Walsh \cite{walsh1981oip}.
\end{remark}

\begin{remark}
Conditions (C1) and (C2) together imply that for any $s \leq t$,
$|\straT_i(t) - \straT_i(s)| \leq t - s$. It follows that the
measure $dC_i$ is absolutely continuous and so it makes sense
to talk about the \textit{rate} $\dot\straT_i(t) = d\straT_i(t)/dt$,
$t \geq0$,
at which $X_i$ is to be run.
\end{remark}

The interpretation is that $\straT_i(t)$ models the total amount of
time spent running $X_i$ by calendar time $t$, and accordingly, the
\textit{controlled process} $X^\straT$ is defined
by
\[
X^\straT(t) \eqdef(X_1(\straT_1(t)), X_2(\straT_2(t)), X_3(\straT
_3(t))),  \qquad   t \geq0.
\]

Continuity of $\straT$ implies that $X^\straT$ is a continuous
process in $\SS$. It is adapted to the
(one parameter) filtration $\FiltMP^\straT$ defined by
\[
\Filt^\straT(t) \eqdef \bigl\{ F \in\Filt\dvtx  F \cap\{ \straT(t)
\preceq\eta\} \in\Filt(\eta) \mbox{ for every } \eta\in\R^3_+
 \bigr\},  \qquad  t \geq0,
% % \straT_1(t) \leq v_1, \straT_2(t) \leq v_2, \straT_3(t) \leq v_3
\]
which satisfies the usual conditions. %(see Lemma \ref{l:rightctyFc}).

The \textit{decision time} $\tau^\straT$ for a time allocation
strategy $\straT$ is
the first time that $X^\straT$ hits the \textit{decision set}
\[
\SSd \eqdef \bigl\{(x_1,x_2,x_3) \in\SS\dvtx  x_{i} = x_{j} \in\{0,1\}
 \mbox{ for some } 1 \leq i < j \leq3 \bigr\}.
\]

The objective is to find a strategy whose associated decision time is a
stochastic minimum.
Clearly, it is possible to do very badly by only
ever running one of the processes as a decision may never be reached
(these strategies
do not need to be ruled out in our model).
A more sensible thing to do is to pick two of the processes, and run
them until
they are absorbed. Only if they disagree do we run the third. This
strategy is much better
than the pathological one (the decision time is almost surely finite!)
but we can do better.

We do not think it is obvious what the best strategy is.
In the situation that $X_1(0)$ is close to zero and $X_3(0)$ close to
one, it is probable that $X_1$ and $X_3$ will be absorbed at different
end points of $[0,1]$.
So, if $X_2(0)$ is close to $0.5$ say, it seems likely that $X_2$ will be
pivotal and so we initially run it, even though $X_1$ and $X_3$ might be
absorbed much more quickly. Our guess is to run the diffusion
whose value lies between that of the other two processes. But if all the
processes are near one, it is not at all clear that this is the best
thing to do.
For example, one could be tempted to run the process with largest value
in the hope that it will give a decision very quickly.

It turns out that we must always ``run the middle.'' That is, if, at
any moment $t \geq0$, we have
$X^\straT_1(t) < X_2^\straT(t) < X^\straT_3(t)$, then we should run
$X_2$ exclusively until it hits $X^\straT_1(t)$ or $X^\straT_3(t)$.
We need not concern ourselves with what happens when the processes are
equal. This is because there is, almost surely, only one strategy that
runs the middle of the three
diffusions when they are separated.
To state this result, let us say that for a strategy $\straT$,
component $\straT_i$ increases at time
$t \geq0$ if $\straT_i(u) > \straT_i(t)$ for every $u > t$.

\begin{lemma}\label{l:stratexists}
There exists a time allocation strategy $\straTopt$
with the property that
(RTM) for each $i \in V$, $\straT^\star_i$ increases at
time $t \geq0$
% $\straT^\star_i$ increases only at times $t \geq0$ %(i.e. )
only if
\[
X_j^{\straTopt}(t) \leq X_i^{\straTopt}(t) \leq X_k^{\straTopt}(t)
\]
for some choice $\{j,k\} = V-\{i\}$.

If $\straT$ is any other strategy with this property, then
$\straT(t) = \straTopt(t)$ for all $t \geq0$ almost surely (with respect
to any of the measures $\Pr_x$).
\end{lemma}

This lemma is proved in Section \ref{s:stratexists} and Theorem \ref
{t:stochminimality} states
that $\straTopt$ gives a stochastic minimum for the decision time.

In the sequel, the drift term $\mu$ is assumed to vanish.
This is not a restriction, for if a drift is present we
may eliminate it by rewriting the problem in natural scale.

%s2 ###
\section{The Laplace transform of the distribution of $\tauopt
$}\label{s:laplacetransform}\label{sec2}

The proof of Theorem~\ref{t:stochminimality} begins by computing
the Laplace transform
\[
\hat v_r(x) \eqdef\E_x  [\exp(-r \tauopt)  ],
\]
of the distribution of the decision time. % $\tauopt\equiv\tau^

This nontrivial task is carried out using the ``guess and verify''
method. Loosely, the guess is inspired by comparing the payoffs of doing
something optimal against doing something nearly optimal. This leads to
a surprisingly tractable heuristic equation from which $\hat v_r$ can
be recovered.

The argument which motivates the heuristic proceeds as follows.
From any strategy $\straT$ it is possible to construct (but we omit
the details)
another strategy, $\hat\straT$, that begins by running $X_1$ for some
small time $h > 0$
[i.e., $\hat\straT(t) = (t,0,0)$ for $0 \leq t \leq h$] and then does
not run
$X_1$ again until $\straT_1$ exceeds~$h$, if ever.
In the meantime, $\hat\straT_2$ and $\hat\straT_3$ essentially
follow $\straT_2$ and
$\straT_3$ with the effect that once $\straT_1$ exceeds~$h$, $\straT$
and $\hat\straT$ coincide.

This means that if the amount of time, $\straT_1(\tau^\straT)$, that
$\straT$ spends\vspace*{-1pt}
running $X_1$ is at least $h$, then $\tau^{\hat\straT}$ and $\tau
^{\straT}$ are identical.
On the other hand, if $\straT_1(\tau^\straT) < h$, then $\hat\straT
$ runs $X_1$
for longer than $\straT$, with some of the time $\hat\straT$ spends
running $X_1$ being wasted. In fact, outside a set with probability
$o(h)$ we have
%
%e2.1 ###
\begin{equation}\label{e:tauhat}
\tau^{\hat\straT} = \tau^\straT+  (h - T_1 )^+,
\end{equation}
where $T_i = \straT_i(\tau^\straT)$ is the amount of time that
$\straT$ spends running
$X_i$ while determining the decision.

We compare $\hat\straT$ with the strategy that runs $X_1$ for time
$h$ and
then behaves \textit{optimally}. If we suppose that $\straTopt$ itself
is optimal
and recall that $\hat v_r$ is the corresponding payoff, this yields the
inequality
%
%e2.2 ###
\begin{equation}\label{e:heurineq}
\E_x  [ \exp( -r \tau^{\hat\straT})  ] \leq\E_x
[ \exp(-rh) \hat v_r(X_1(h),X_2(0), X_3(0))  ].
\end{equation}

Now, we take $\straT= \straTopt$ and use \eqref{e:tauhat} to
see that the left-hand side of \eqref{e:heurineq} %\eqref{e:tauhat},
%the left hand side of \eqref{e:heurineq}
is equal to
\[
\E_x  \bigl[ \exp\bigl( -r \bigl( \tau^\straTopt+ (h - T_1)^+\bigr)\bigr)  \bigr] + o(h),
\]
which, in turn, may be written as
%
%e2.3 ###
\begin{equation}\label{e:heur1}
\hat v_r(x) + \E_x \bigl[  \bigl(\exp\bigl(-r(\tauopt+ h)\bigr) - \exp
(-r\tauopt)  \bigr) \Indi{T_i = 0}  \bigr] + o(h).
\end{equation}

On the other hand, if we assume $\hat v_r$ is suitably smooth, the
right-hand side of \eqref{e:heurineq} is
%
%e2.4 ###
\begin{equation}\label{e:heur2}
\hat v_r(x) + h  (\Generator^1 - r  )\hat v_r(x) + o(h), \qquad
x_1 \in(0,1),
\end{equation}
where we have introduced the differential operator $\Generator^i$
defined by
\[
\Generator^i f(x) \eqdef\frac{1}{2}\sigma^2(x_i) \,\frac{\partial^2
}{\partial x_i^2} f(x), \qquad   x_i \in(0,1).
\]

After substituting these expressions back into \eqref{e:heurineq} and
noticing that there was nothing special about choosing $X_1$ to be the
process that we
moved first, we see that
%
%e2.5 ###
\begin{equation}\label{e:heurineq1}
 \qquad \E_x \bigl[  \bigl( \exp\bigl(-r(\tauopt+ h)\bigr) - \exp(-r\tauopt)
\bigr) \Indi{T_i = 0} \bigr] \leq h  (\Generator^i - r  )\hat
v_r(x) + o(h)
\end{equation}
for each $x_i \in(0,1)$ and $i \in V$.

Dividing both sides by $h$, and taking the limit $h \to0$ yields the inequality
%
%e2.6 ###
\begin{equation}\label{e:heur3}
 (\Generator^i - r  )\hat v_r(x)
\leq- r \E_x \bigl[\exp(-r\tauopt)\Indi{T_i = 0} \bigr].
\end{equation}

Now, in some simpler, but nevertheless related problems, we can
show that \eqref{e:heur3} is true with an \textit{equality} replacing the
inequality. This prompts us to try to \textit{construct} a function
satisfying \eqref{e:heur3} with equality. Our effort culminates in
the following.

\begin{lemma}\label{l:heurfnexists} There exists a continuous function
$h_r\dvtx  \SS\to\R$ such that:
\begin{itemize}
\item$h_r(x) = 1$ for $x \in D$,
\item the partial derivatives $\frac{\partial^2 \hat h_r}{\partial
x_i \,\partial x_j}$ exist
and are continuous on $\{ x \in\SS\setminus D \dvtx  x_i, x_j \in(0,1)\}
$ (for any $i,j \in V$ not necessarily distinct) and
\item furthermore, for each $i \in\V$ and $x \notin D$ with $x_i \in(0,1)$,
\[
 (\Generator^i - r  )h_r(x)
= - r \hat f^i_r(x),
\]
where $\hat f^i_r(x) \eqdef\E_x [\exp(-r\tauopt)\Indi{T_i =
0} ]$.
\end{itemize}
\end{lemma}

\begin{pf}
We begin by factorizing $\hat f_r^i(x)$ into a product of Laplace
transforms of
diffusion exit time distributions. This factorization is useful as it
allows us to construct $h$ by solving a series of ordinary
differential equations. Note that in this proof, we will typically
suppress the $r$ dependence for notational convenience.

The diffusions all obey the same stochastic differential equation and so
we lose nothing by assuming that the components of $x$ satisfy
$0 \leq x_1 \leq x_2 \leq x_3 \leq1$. Further, we suppose that $x
\notin D$
because otherwise $T_i = 0$ $\Pr_x$-almost-surely.

In this case, $T_2 > 0$ $\Pr_x$-almost-surely, because for any $t >
0$, there exist times $t_1, t_3 < t/2$
at which $X_1(t_1) < x_1 \leq x_2 \leq x_3 < X_3(t_3)$ and so it is
certain our
strategy allocates time to $X_2$. It follows that $\hat f^2(x)$ vanishes.

%So, as , we lose nothing by assuming that $x_1 < x_2 \leq x_3 \leq1$.
Now consider $\hat f^1$. There is a $\Pr_x$-negligible set off which
$T_1 = 0$ occurs if, and only if,
both of the independent diffusions $X_2$ and $X_3$ exit the interval
$(X_1(0),1)$ at the upper boundary. Furthermore, $\tauopt$ is just
the sum of the exit times. That is, if
%
%e2.7 ###
\begin{equation}\label{e:xihita}
\exitST^{(i)}_a \eqdef\inf\{t > 0\dvtx  X_i(t) = a\}, \qquad    a \in[0,1],\ i
\in V,
\end{equation}
then
\[
\hat f^1(x) = \E_x \bigl[\exp \bigl(-r  \bigl(\exitST^{(2)}_1 +
\exitST^{(3)}_1 \bigr) \bigr)
\Indi{ \exitST^{(2)}_1 < \exitST^{(2)}_{x_1}, \exitST^{(3)}_1 <
\exitST^{(3)}_{x_1}} \bigr].
\]

Using independence of $X_2$ and $X_3$, we have the factorization
\[
\hat f^1(x) = \prod_{i=2}^3 \E_x \bigl[\exp \bigl(-r \exitST
^{(i)}_1 \bigr)\Indi{\exitST^{(i)}_1 < \exitST^{(i)}_{x_1}} \bigr].
\]
Note that our assumption $x \notin D$ guarantees that $x_1 < 1$.

To write this more cleanly, let us introduce, for $0 \leq a < b \leq1$,
% $h^+_{a,b}:[a,b] \to[0,1]$
the functions
\[
h^+_{a,b}(u) \eqdef\E_{u} \bigl[\exp \bigl(-r \exitST^{(1)}_b
\bigr)\Indi{\exitST^{(1)}_b < \exitST^{(1)}_{a}} \bigr],
\]
where the expectation operator $\E_u$ corresponds to the (marginal)
law of
$X_1$ when it begins at $u \in[0,1]$. The diffusions obey the same
SDE, and so
%
%e2.8 ###
\begin{equation}\label{e:fh1}
\hat f^1(x) = h^+_{x_1,1}(x_2)h^+_{x_1,1}(x_3).
\end{equation}
Similarly,
%
%e2.9 ###
\begin{equation}\label{e:fh3}
\hat f^3(x) = h^-_{0,x_3}(x_1)h^-_{0,x_3}(x_2),
\end{equation}
where
\[
h^-_{a,b}(u) \eqdef\E_u \bigl[\exp \bigl(-r \exitST^{(i)}_a
\bigr)\Indi{\exitST^{(i)}_a < \exitST^{(i)}_{b}} \bigr].
\]

We take, as building blocks for the construction of $h$, the functions
$h^\pm_{0,1}$, abbreviated to $h^\pm$ in the sequel.
%The regularity of each of our (nonsingular) diffusions together with
%the
If $a < b$ and $u \in[a,b]$ then by the strong Markov property,
\[
h^+(u) = h^+_{a,b}(u)h^+(b) + h^-_{a,b}(u)h^+(a)
\]
and
\[
h^-(u) = h^+_{a,b}(u)h^-(b) + h^-_{a,b}(u)h^-(a).
\]
Solving these equations gives
%
%e2.10 ###
\begin{equation}\label{e:hpab}
h^+_{a,b}(u) = \frac{h^-(a)h^+(u) - h^-(u)h^+(a)}{h^-(a)h^+(b)- h^-(b)h^+(a)}
\end{equation}
and
%
%e2.11 ###
\begin{equation}\label{e:hmab}
h^-_{a,b}(u) = \frac{h^-(u)h^+(b) - h^-(b)h^+(u)}{h^-(a)h^+(b)- h^-(b)h^+(a)}.
\end{equation}

The functions $h^+$ and $h^-$ are $C^2$ on $(0,1)$ and continuous on $[0,1]$.
Furthermore, they solve $\Generator f = rf$ where $\Generator f \eqdef
\frac{1}{2}\sigma^2(\cdot) f^{\prime\prime}$.
In light of this, and remembering our assumption that the components of
$x$ are ordered,
we will look for functions $\lambda^+$ and $\lambda^-$ of $x_1$ and
$x_3$ such that
%
%e2.12 ###
\begin{equation}\label{e:hinlambda}
h(x) = \lambda^-(x_1,x_3)h^-(x_2) + \lambda^+(x_1,x_3)h^+(x_2)
\end{equation}
has the desired properties. For other values of $x \notin D$, we will
define $h$ by symmetry.\vadjust{\goodbreak}

To get started, plug \eqref{e:hpab} and \eqref{e:hmab} into \eqref
{e:fh1} and \eqref{e:fh3}
to see that $\hat f^i(x)$ has a linear
dependence on $h^+(x_2)$ and $h^-(x_2)$, that is,  % there are functions
%$\psi_\pm^i$ such that
%
\[
\hat f^i(x) = \psi^i_-(x_1,x_3)h^-(x_2) + \psi^i_+(x_1,x_3)h^+(x_2),
\]
%
% For example,
where
\begin{eqnarray*}
\psi^1_+(x_1,x_3) &\eqdef&\frac{h^-(x_1)h^+(x_3) - h^-(x_3)h^+(x_1)
}{h^-(x_1)},
\\
\psi^1_-(x_1,x_3) &\eqdef&- \frac{h^+(x_1)}{h^-(x_1)} \psi^1_+(x_1,x_3),
\\
\psi_-^3(x_1,x_3) &\eqdef&\frac{h^-(x_1)h^+(x_3) - h^-(x_3)h^+(x_1)
}{h^+(x_3)},
\end{eqnarray*}
and
\[
\psi_+^3(x_1,x_3) \eqdef-\frac{h^-(x_3)}{h^+(x_3)}\psi_+^3(x_1, x_3).
\]

Linearity of the operator $ ( \Generator^i - r  )$ and
linear independence of
$h^-$ and $h^+$ then show the requirement that $ (\Generator^i -
r )h = -r\hat f^i$
boils down to requiring
\[
 ( \Generator^i - r )\lambda_\pm= -r\psi^i_\pm.
\]

Of course, the corresponding homogeneous equations are solved with
linear combinations of $h^+$ and $h^-$---what remains is the essentially
computational task of finding particular integrals and some constants.

This endeavour begins with repeated application of Lagrange's variation
of parameters
method, determining constants using the boundary conditions $h(x) = 1$
for $x \in D$ where possible. Eventually, we are left wanting
only for real constants, an unknown function of $x_1$ and a function of
$x_3$. At this point, we appeal
to the ``smooth pasting'' conditions
%
%e2.13 ###
\begin{equation}\label{e:smoothpasting}
  \biggl( \frac{\partial}{\partial x_i} -\, \frac{\partial
}{\partial x_j}  \biggr)h\bigg |_{x_i = x_j} = 0,  \qquad   i,j \in V.
\end{equation}

After some manipulation, we are furnished with differential equations
for our unknown functions
and equations for the constants. These we solve with little difficulty
and, in doing so,
determine that
\begin{eqnarray*}
\lambda_-(x_1,x_3) &=& h^-(x_1) - h^+(x_1)h^+(x_3)\int_{x_3}^1 \frac
{\fracc{d}{d u} h^-(u)}{h^+(u)^2}\,du \\
& &{} + h^-(x_1) h^+(x_3) \int_{0}^{x_1} \frac{\fracc{d}{d u}
h^+(u)}{h^-(u)^2}\,du \\
& &{} + \frac{2r h^-(x_3) }{\phi} \int_0^{x_1} \biggl( \frac
{h^+(u)}{\sigma(u)h^-(u)}  \biggr)^2 \\
&&\hphantom{{}+ \frac{2r h^-(x_3) }{\phi} \int_0^{x_1}}
{}\times \bigl( h^-(x_1)h^+(u)-
h^-(u)h^+(x_1)  \bigr)\,du,
\end{eqnarray*}
and
\begin{eqnarray*}
\lambda_+(x_1,x_3) &= & h^+(x_3) + h^-(x_1) h^-(x_3) \int_{0}^{x_1}
\frac{\fracc{d}{d u} h^+(u)}{h^-(u)^2}\,du \\
& &{} - h^-(x_1)h^+(x_3)\int_{x_3}^1 \frac{\fracc{d}{d u}
h^-(u)}{h^+(u)^2}\,du \\
&&{} + \frac{2r h^+(x_1) }{\phi} \int_{x_3}^{1} \biggl( \frac
{h^-(u)}{\sigma(u)h^+(u)}  \biggr)^2\\
&&\hphantom{{}+ \frac{2r h^+(x_1) }{\phi} \int_{x_3}^{1}}
{}\times  \bigl( h^-(u)h^+(x_3) -
h^-(x_3)h^+(u)  \bigr)\,du,
\end{eqnarray*}
where $\phi$ denotes the constant value of the Wronskian $h^-(u)\frac
{d}{d u} h^+(u) -  h^+( u)\times\break\frac{d}{d u} h^-(u)$.

These expressions for $\lambda^\pm$ are valid for any $x$ not lying
in $\SSd$ with weakly ordered components; so $h$ is defined outside of
$D$ via \eqref{e:hinlambda}. Naturally,
we define $h$ to be equal to one on $D$.

Having defined $h$, we now show that it is continuous and has the
required partial derivatives.
Continuity is inherited from $h^+$ and $h^-$ on the whole of $\SS$
apart from at
the exceptional corner points $(0,0,0)$ and $(1,1,1)$ in $D$. For these
two points, a~few lines of
justification are needed. We shall demonstrate continuity at the
origin, continuity at the upper right-hand
corner $(1,1,1)$ follows by the same argument. Let $x^n$ be a sequence
of points in $\SS$ that converge
to $(0,0,0)$; we must show $h(x^n) \to h(0,0,0) = 1$. Without loss
of generality, assume that
the components of $x^n$ are ordered $x^n_1 \leq x^n_2 \leq x^n_3$ and that
$x^n$ is not in $D$ [if $x^n \in D$, then $h(x^n) = 1$ and it may be
discarded from the sequence].
From the expression \eqref{e:hinlambda} for $h$, we see that it is
sufficient to check that
\[
\mbox{(i)} \quad    \lambda^-(x^n_1, x^n_3) \to1
 \quad \mbox{and}  \quad   \mbox{(ii)} \quad
h^+(x^n_2)\lambda^+(x^n_1, x^n_3) \to0,
\]
since $h^-(x^n_2) \to1$. For (i), the only doubt is that the term
involving the first integral in the expression
for $\lambda^-$ does not vanish in the limit. The fact that it
does can be proved by the Dominated Convergence theorem. The term is
\[
h^+(x^n_1)h^+(x^n_3) \int_{x^n_3}^1 \frac{\fracc{\partial}{\partial
u} h^-(u)}{h^+(u)^2}\,du
= \int_{0}^1 \Indi{u > x^n_3} \frac{h^+(x^n_1)h^+(x^n_3)}{h^+(u)^2}
\frac{\partial}{\partial u} h^-(u)\,du.
%  | h^+(x^n_1)h^+(x^n_3) \int_{x^n_3}^1 \frac{\frac{\partial}{
% | \\
% &\leq& h^-(x^n_3) - h^-(1).
\]
The ratio $\frac{h^+(x^n_1)h^+(x^n_3)}{h^+(u)^2}$ is bounded above by
one when $u > x^n_3 \geq x^n_1$
since $h^+$ is increasing. Further, the derivative of $h^-$ is
integrable and so
the integrand is dominated by an integrable function, and converges to zero.

For the second limit (ii), there are two terms to check. First, that
\[
h^+(x^n_2)h^-(x^n_1)h^+(x^n_3)\int_{x^n_3}^1 \frac{\fracc{\partial
}{\partial u} h^-(u)}{h^+(u)^2}\,du \to0
\]
follows from essentially the same argument as before. The second term
of concern is
\[
h^+(x^n_1)\int_{x^n_3}^{1} \biggl( \frac{h^-(u)}{\sigma(u)h^+(u)}
 \biggr)^2  \bigl( h^-(u)h^+(x^n_3) - h^-(x^n_3)h^+(u)  \bigr)\,du.
\]

Again, one may write this as the integral of a dominated function (recalling
that $\sigma$ is bounded away from zero) that converges to zero. Thus,
the integral above converges to zero as required.

Now that we have established continuity of $h$, we can begin tackling
the partial derivatives.

When the components of $x$ are distinct, differentiability comes from
that of our building blocks $h^+$
and $h^-$. It is at the switching boundaries, when two or more
components are equal,
where we have to be careful. The key here is to remember that we
constructed $h$ to
satisfy the smooth pasting property \eqref{e:smoothpasting}---this
allows us to show that the one-sided partial derivatives are equal at
the switching boundaries. For example, provided the limit exists,
\[
 \frac{\partial}{\partial x_1}h(x_1,x_2,x_3) \bigg|_{x_1 =
x_2=x_3} =
\lim_{\epsilon\to0} \frac{1}{\epsilon}\bigl ( h(x_1 + \epsilon,
x_1,x_1) - h(x_1,x_1,x_1)  \bigr).
\]
Using \eqref{e:hinlambda} and the differentiability of $\lambda$, the
limit from above is
\[
 \frac{\partial}{\partial x_3} \bigl( \lambda
^-(x_1,x_3)h^-(x_2) + \lambda^+(x_1,x_3)h^+(x_2)  \bigr)\bigg |_{x_1
= x_2=x_3}.
\]
This is equal to the limit from below,
\[
 \frac{\partial}{\partial x_1} \bigl( \lambda
^-(x_1,x_3)h^-(x_2) + \lambda^+(x_1,x_3)h^+(x_2)  \bigr)\bigg |_{x_1
= x_2=x_3},
\]
by the smooth pasting property. % (the reader may verify it directly).
The other first-order partial derivatives exist by similar arguments.
Note that we do not include in our hypothesis the requirement that
these first-order partial
derivatives exist at the boundary points of the interval.

The second-order derivatives are only slightly more laborious to check.
As before it is at switching boundaries where we must take care in
checking that
the limits from above and below agree. For the partial derivatives
$\frac{\partial^2}{\partial x_i^2}h$\vspace*{-2pt} at a point $x$ not in $D$ with
$x_i \in(0,1)$,
we equate the limits using the fact that $(\Generator^i - r)h(x)$
vanishes whenever
$x_i$ is equal to another component of $x$ rather than smooth pasting.
For the mixed partial derivatives, we use a different argument.
When exactly two components are equal, there is no problem. This
is a consequence of the form \eqref{e:hinlambda} of $h$---one
component enters through
the terms $h^+$ and $h^-$ while the other two components % term that
enter through $\lambda^+$ and $\lambda^-$. For example, if $x_1 = x_2
< x_3$, then % by \eqref{e:hinlambda}
% due to the form
% \eqref{e:hinlambda}
%
\begin{eqnarray*}
  \frac{\partial^2 }{\partial x_1 \,\partial x_2}
h(x_1,x_2,x_3) \bigg|_{x_1=x_2} &=&
\biggl (\frac{dh^-}{dx_1}(x_1) \biggr)\, \frac{\partial}{\partial
x_1}\lambda^-(x_1,x_3)\\
&&{} +
 \biggl(\frac{dh^+}{dx_1}(x_1) \biggr)\, \frac{\partial}{\partial
x_1}\lambda^+(x_1,x_3)
\end{eqnarray*}
regardless of how the switching boundary is approached. When all three
components are
equal, we must check that
\begin{eqnarray*}
 \frac{\partial^2 h }{\partial x_1\, \partial
x_3}(x_1,x_2,x_3) \bigg|_{x_1 = x_2 = x_3} &=&
 \frac{\partial^2 h }{\partial x_2 \,\partial
x_3}(x_1,x_2,x_3) \bigg|_{x_1 = x_2 = x_3}\\
 &=&
 \frac{\partial^2 h }{\partial x_1\, \partial
x_2}(x_1,x_2,x_3) \bigg|_{x_1 = x_2 = x_3}.
\end{eqnarray*}
This is straightforward to do. Thus, $h$ has all of the properties we required.
\end{pf}

From here, we need a verification lemma to check that the function we
constructed really is equal
to $\hat v_r$. The following result does just that, and, as a
corollary, shows that $\hat v_r$ is maximal among Laplace transforms of
decision time distributions
(note that this is weaker than the stochastic minimality claimed in
Theorem \ref{t:stochminimality}).
The result is essentially that Bellman's principle of optimality holds
(specialists in optimal control will notice that
the function we constructed in Lemma \ref{l:heurfnexists}
satisfies the Hamilton--Jacobi--Bellman equation).

\begin{lemma}\label{l:vhatverification} Suppose that $h_r\dvtx  \SS\to\R
$ satisfies:
\begin{itemize}
\item$h_r$ is continuous on $\SS$,
% \item$\frac{\partial^2 h_r}{\partial x_i \partial x_j}$ exists is
%continuous at
% any $x \in\SS-D$ with $x_i, x_j \in(0,1)$, $i,j \in V$.
%
\item for $i,j \in V$, $\frac{\partial^2 h_r}{\partial x_i \,\partial
x_j}$ exists and is continuous on $\{ x \in\SS\setminus\SSd\dvtx  x_i,
x_j \in(0,1)\}$,
\item$h_r(x) = 1$ for $x \in D$,
\item$ (\Generator^i - r )h_r(x) \leq0$.
\end{itemize}
Then
\[
h_r(x) \geq\sup_{\straT} \E_x [\exp(-r \tau^{\straT}) ].
\]

Furthermore, if $ (\Generator^i - r )h_r(x)$ vanishes
whenever $x_j \leq x_i \leq x_k$ (under some labeling) then
\[
h_r(x) = \hat v_r(x) = \E_x [\exp(-r \tauopt) ].
\]
\end{lemma}

\begin{pf}
Let $\straT$ be an arbitrary strategy and define\vspace*{-1pt} the function $g\dvtx \SS
\times[0,\infty) \to\R$
by $g(x,t) \eqdef\exp(-rt)h_r(x)$. Then, by hypothesis,\vspace*{-1pt} $g$ is
$C^{2,1}$ on
$(0,1)^3 \times[0,\infty)$.
Thus, if $\operatorname{dist}$ denotes Euclidean distance and $\rho_n \eqdef\inf\{t
\geq0\dvtx\break   \operatorname{dist}(X^\straT(t), \partial\SS) < n^{-1}\}$, It{\^o}'s
formula shows that
\begin{eqnarray*}
g (X^\straT(\rho_n),\rho_n ) - g(X^\straT(0),0) &=& \sum_i \int
_0^{\rho_n} \,\frac{\partial}{\partial x_i}g(X^\straT(s),s)\,dX^\straT
_i(s) \\
&&{} + \int_0^{\rho_n} \,\frac{\partial}{\partial t}g(X^\straT
(s),s)\,ds\\
&&{} + \frac{1}{2} \sum_{i,j} \int_0^{\rho_n} \,\frac{\partial
^2}{\partial x_i\,\partial x_j}g(X^\straT(s),s)\,d[X^\straT_i,X^\straT_j]_s.
\end{eqnarray*}

Theorem \ref{t:MPCSM} implies
$[X^\straT_i]_s = [X_i]_{\straT_i(s)}$ and that
$X^\straT_i$ and $X^\straT_j$ are orthogonal martingales. Hence,
using absolute continuity of $\straT$ and Proposition 1.5, Chapter~V
of \cite{revuzyor},
\begin{eqnarray*}
g(X^\straT(\rho_n),\rho_n) - g(X^\straT(0),0) &=& \sum_i \int
_0^{\rho_n} \,\frac{\partial}{\partial x_i}g(X^\straT(s),s)\,dX^\straT
_i(s) \\
&&{} + \sum_i \int_0^{\rho_n} \exp(-rs)  (\Generator^i - r
 ) h(X^\straT(s)) \dot\straT_i(s)\,ds.
%&& -r \int_0^{\rho_n} \exp(-rs) h(X^\straT(s))\,ds
\end{eqnarray*}

The integrand of the stochastic integral against the square integrable
martingale $X^\straT_i$ is continuous and hence bounded on each
compact subset of $(0,1)^3$.
Thus, the integral's expectation vanishes, that is,
\[
\E_x \biggl[ \int_0^{\rho_n} \,\frac{\partial}{\partial
x_i}g(X^\straT(s),s)\,dX^\straT_i(s)  \biggr] = 0.
\]

Next, the fact that $ (\Generator^i -r  ) h$ is not positive
gives
\[
\E_x  \biggl[\int_0^{\rho_n} \exp(-rs)  (\Generator^i -r
 ) h(X^\straT(s)) \dot\straT_i(s)\,ds \biggr] \leq0,
\]
and so
%
%e2.14 ###
\begin{equation}\label{e:lemver1}
\E_x  [\exp(-r\rho_n)h(X^\straT(\rho_n)) ] - h(x) %=
%h(X^\straT(s)) \dot\straT_i(s)\,ds.
\leq0.
\end{equation}

Now, the times $\rho_n$ taken for $X^\straT$ to come within distance $n^{-1}$
of the boundary of $\SS$ converge to $\rho\eqdef\inf\{t \geq0\dvtx
X^\straT(t) \in\partial\SS\}$
as $n \to\infty$. So, the continuity of $h$ and the Dominated
Convergence theorem together imply
%
%e2.15 ###
\begin{equation}\label{e:lemver2}
\E_x  [\exp(-r\rho)h(X^\straT(\rho)) ] \leq h(x).
\end{equation}

In summary, inequality \eqref{e:lemver2} arises by applying
the three dimensional It{\^o} formula
to $g$ composed with the controlled process stopped inside $(0,1)^3$ and
then using continuity of $h$. But, from time $\rho$ onward, our
controlled process
runs on a face or an edge of the cube and It{\^o}'s formula in three dimensions
does not apply. This is not a problem though---a similar
argument with It{\^o}'s formula in one (or two) dimensions does the trick.
That is,
if $\rho^\prime$ denotes the first time that $X^\straT$ hits an edge
of $\SS$
(so $0 \leq\rho\leq\rho^\prime\leq\tau^\straT$),
then both
%e2.16 ###
\begin{equation}\label{e:lemver3}
\E_x  [\exp(-r\rho^\prime)h(X^\straT(\rho^\prime)) - \exp
(-r\rho)h(X^\straT(\rho)) ] \leq0
\end{equation}
and
%
%e2.17 ###
\begin{equation}\label{e:lemver4}
\E_x  [\exp(-r\tau^\straT)h(X^\straT(\tau^\straT)) - \exp
(-r\rho^\prime)h(X^\straT(\rho^\prime)) ] \leq0.
\end{equation}

Summing these differences and using the boundary condition $h(x) = 1$ for
$x \in D$ yields %$h(X^\straT(\tau^\straT)) = 1$ yields
\[
\E_x  [\exp(-r\tau^\straT) ] = \E_x  [\exp(-r\tau
^\straT)h(X^\straT(\tau^\straT)) ]
\leq h(x).
\]

Thus, $h$ is an upper bound for the Laplace transform of the
distribution of
the decision time arising from any strategy.
It remains to prove that $h$ is equal to the Laplace transform $\hat
v_r$.% the decision time arising from the strategy in Lemma

Suppose that $\straT$ is the strategy $\straTopt$ from Lemma \ref
{l:stratexists}, then for almost every \mbox{$s \geq0$}, $\dot\straT_i(s)$
is positive only when $X^\straT_j(s) \leq X^\straT_i(s) \leq X^\straT
_k(s)$ under some labeling.\vspace*{-2pt}
So, $ (\Generator^i -r  ) h(X^\straT(s))\dot\straT_i(s)$
vanishes for almost every $s \geq0$ and \eqref{e:lemver1} is an equality.
Taking limits show that \eqref{e:lemver2}--\eqref{e:lemver4}
are also equalities.
\end{pf}

So, $\hat v_r$ is twice differentiable in each component and satisfies
the heuristic
equation
%
%e2.18 ###
\begin{equation}\label{e:heureq4}
 (\Generator^i - r  )\hat v_r(x)
= - r \hat f^i_r(x),  \qquad   x \notin D, \ x_i \in(0,1).
\end{equation}

In the next section, we will show that $\Pr_x(\tauopt> t)$ is the
probabilistic
solution to certain parabolic partial differential equations. To do
this, we
need to rewrite $\hat v_r$ in a more suitable form. Introduce the notation
$X^{(1)}(t) = (X_1(t),X_2(0),X_3(0))$,
$X^{(2)}(t) = (X_1(0),X_2(2),X_3(0))$ and
$X^{(3)}(t) = (X_1(0),\break X_2(0),X_3(t))$ for each $t \geq0$.
We define $\rho^{(i)}$ to be
the absorption time of
$X_i$, that is,
\[
\rho^{(i)} \eqdef\inf\{t \geq0\dvtx  X_i(t) \notin(0, 1)\}.
\]

\begin{lemma}\label{l:vrhatrep}
For any $x \notin D$, $\hat v_r$ can be written as
\[
\hat v_r(x) = \E_x  \biggl[ \exp\bigl(-r \rho^{(i)}\bigr)
 \hat v_r\bigl(X^{(i)}\bigl(\rho^{(i)}\bigr)\bigr)
+r%\int_0^\rho ( \int_0^\infty f(u,X^{(i)}(s))\exp(-ru)\,du  )
\int_0^{\rho^{(i)}} \hat f^i_r\bigl(X^{(i)}(s)\bigr) \exp(-rs)\,ds  \biggr].
\]
\end{lemma}

\begin{pf}
Fix $x \notin D$, then the function $x_i \mapsto\hat v_r(x)$ is $C^2$
on $(0,1)$ and $C^0$ on $[0,1]$.
Introduce the a.s. finite $\Filt_i$ stopping time $\rho^{(i)}_n
\eqdef\inf\{t \geq0\dvtx  X_i(t) \notin(n^{-1}, 1-n^{-1})\}$, so It{\^o}'s
formula (in one dimension) gives
\begin{eqnarray*}
&&\exp\bigl(-r\rho^{(i)}_n\bigr)\hat v_r\bigl(X^{(i)}\bigl(\rho^{(i)}_n\bigr)\bigr)
 - \hat v_r(X(0))\\
&& \qquad  =
\int_0^{\rho^{(i)}_n} \exp(-rs) \,\frac{\partial}{\partial x_i} \hat
v_r\bigl(X^{(i)}(s)\bigr)\,dX_i(s) \\
&& \qquad  \quad {} + \int_0^{\rho^{(i)}_n} \exp(-rs) ( \Generator^i - r
) \hat v_r\bigl(X^{(i)}(s)\bigr)\,ds.
\end{eqnarray*}

The function $\frac{\partial}{\partial x_i} \hat v_r$ is continuous
on $(0,1)$ and hence bounded on the compact
sets $[n^{-1}, 1-n^{-1}]$. It follows that
the expectation of the stochastic integral against $dX_i$
vanishes. % taking expectations with respect to $\Pr_x$ and
So, using equation \eqref{e:heureq4},
\begin{eqnarray*}
\hat v_r(x) & = & \E_x \bigl[ \exp\bigl(-r\rho^{(i)}_n\bigr)\hat
v_r\bigl(X^{(i)}\bigl(\rho^{(i)}_n\bigr)\bigr) \bigr]
\\
&& {}+r\E_x \biggl[\int_0^{\rho^{(i)}_n} \exp(-rs)\hat
f^i_r\bigl(X^{(i)}(s)\bigr)\,ds \biggr].
\end{eqnarray*}

The stopping times $\rho^{(i)}_n$ converge to $\rho^{(i)}$ as $n \to
\infty$ and so by continuity of $X_i$,
$\hat v_r$, the exponential function and the integral,
\[
\exp\bigl(-r\rho^{(i)}_n\bigr)\hat v_r\bigl(X^{(i)}\bigl(\rho^{(i)}_n\bigr)\bigr)
 \to\exp\bigl(-r\rho
^{(i)}\bigr)\hat v_r\bigl(X^{(i)}\bigl(\rho^{(i)}\bigr)\bigr)
\]
and
\[
\int_0^{\rho^{(i)}_n} \exp(-rs)\hat f^i_r\bigl(X^{(i)}(s)\bigr)\,ds \to\int
_0^{\rho^{(i)}} \exp(-rs)\hat f^i_r\bigl(X^{(i)}(s)\bigr)\,ds.
\]

To finish the proof, use the Dominated Convergence theorem to exchange
the limit and expectation.
\end{pf}

\begin{remark}\label{r:Etau}
We can generalize our heuristic argument to value functions of the form
% There is a generalised heuristic equation for value functions of the
%form
%
\[
J(x,t) \eqdef\E_x [g(\tauopt+ t) ], \qquad   x \in\SS,\   t
\geq0,
\]
for differentiable $g$. The heuristic equation reads
%
%e2.19 ###
\begin{equation}\label{e:heurgeneric}
 \biggl(\Generator^i + \,\frac{\partial}{\partial t} \biggr) J(x,t) =
\E_x \bigl[g^\prime(\tauopt+ t)\Indi{T_i = 0} \bigr].
\end{equation}

Equation \eqref{e:heureq4} is the specialization $g(t) = \exp(-rt)$.
Such a choice of $g$ is helpful because it effectively removes the time
dependence in \eqref{e:heurgeneric},
making it easier to solve. The benefit is the same if $g$ is linear
and it is not difficult to construct and verify (as we did in
Lemmas \ref{l:heurfnexists} and \ref{l:vhatverification})
an\vspace*{-1pt} explicit expression for $J(x) \eqdef\E_x [\tauopt ]$.
In terms of the expected absorption times $G(u) = \E_u[\exitST
^{(1)}_{0} \wedge\exitST^{(1)}_{1}]$ and integrals
\[
%I_{j,k}(i) = \int_0^i \frac{u^k G(u)}{(1-u)^j}\,du
I_k(x_1) \eqdef\int_0^{x_1} \frac{G(u)}{(1-u)^k}\,du  \quad \mbox{and}
\quad
J_k(x_3) \eqdef\int_{x_3}^1 \frac{G(u)}{u^k}\,du, \qquad   k \in\N,
\]
the expression for $J$ reads
\begin{eqnarray*}
J(x) & = & G(x_2) + (1-x_1)^{-2}G(x_1)\\
&&\hphantom{G(x_2) +}{}\times \bigl(
(1-x_2)\bigl((1-x_1)-(1-x_3)\bigr)+(1-x_1)(1-x_3) \bigr) \\
&&{} -2I_3(x_1) \bigl( (1-x_2)\bigl((1-x_1)+(1-x_3)\bigr)+(1-x_1)(1-x_3)\bigr)
\\
&&{}+ 6I_4(x_1)(1-x_2)(1-x_1)(1-x_3) + x_3^{-2}G(x_3) \bigl(x_2(x_3-x_1) +
x_1 x_3  \bigr) \\
&&{} -2J_3(x_3) \bigl( x_2(x_3+x_1) + x_1 x_3  \bigr) + 6J_4(x_3)x_1 x_2 x_3.
\end{eqnarray*}
\end{remark}

%s3 ###
\section{A representation for $\Pr_x (\tauopt> T )$}\label{sec3}

The aim of this section is to connect the tail probability $v\dvtx \SS
\times[0,\infty) \to[0,1]$ defined by
\[
v(x,t) \eqdef\Pr_x (\tauopt> t ),
\]
to the formula for $\hat v_r$ from Lemma \ref{l:vrhatrep}.
Before continuing, let us explain the key idea. Just for a moment,
suppose that $v$ is smooth
and consider the Laplace transform of $ (\Generator^i - \,\frac
{\partial}{\partial t}  )v(x,\cdot)$.
It is straightforward to show that the Laplace transform of $v$
satisfies [see \eqref{e:laplacev}]
\[
\int_0^\infty v(x,t) \exp(-rt)\,dt = r^{-1} \bigl ( 1 - \hat
v_r(x) \bigr).
\]

Bringing $\Generator^i$ through the integral and integrating by parts
in $t$,
\[
\int_0^\infty\exp(-rt) \biggl (\Generator^i - \,\frac{\partial
}{\partial t}  \biggr)v(x,t)\,dt = -r^{-1}  ( \Generator^i -
r ) \hat v_r(x).
\]
Combining this with the heuristic equation \eqref{e:heureq4} gives
%
%e3.1 ###
\begin{equation}\label{e:genddtvxtLT}
\int_0^\infty\exp(-rt)  \biggl(\Generator^i - \,\frac{\partial
}{\partial t}  \biggr)v(x,t)\,dt =
\hat f_r^i(x).
\end{equation}
This shows that $ (\Generator^i - \,\frac{\partial}{\partial t}
 )v$ is nonnegative
(i.e., $v$ satisfies the associated Hamilton--Jacobi--Bellman equation).
From here,
one could use It{\^o}'s formula (cf.  the proof of Lemma \ref
{l:vhatverification}) to see that
$ (v(X^\straT(t), T-t), 0 \leq t \leq T )$ is a
sub-martingale for any strategy $\straT$.
In particular,
\[
\Pr_x(\tau^{\straT}>T) = \E_x [v(X^\straT(T), 0) ] \geq v(x,T).
\]

So, ideally, to prove Theorem \ref{t:stochminimality}, we would
establish that $v$ is
smooth enough to apply It{\^o}'s formula.
We are given some hope, by noticing that if we\vspace*{1pt} can show that $\hat
f_r^i(x)$ is the Laplace transform of a function $f_i(x,t)$ say, then
\eqref{e:genddtvxtLT} implies that $v$ solves
%
%e3.2 ###
\begin{equation}\label{e:genddtvxt}
 \biggl(\Generator^i - \,\frac{\partial}{\partial t}  \biggr)v = f_i.
\end{equation}
We can show such a density $f_i$ exists (Lemma \ref{l:densityexists}
below) but
not that it is H\"{o}lder continuous. Unfortunately, without the latter,
we cannot show that \eqref{e:genddtvxt} has a classical solution. Nevertheless,
we can deduce the sub-martingale inequality by showing merely that $v$ solves
\eqref{e:genddtvxt} in a weaker sense (Lemma \ref{l:rep}).

To commence, let us first verify the claim that $\hat f_r^i$ is the
Laplace transform of a function.

\begin{lemma}\label{l:densityexists} For each $x \notin D$ and $i \in
V$, the Borel measure $B \mapsto\Pr_x(\tauopt\in B,  T_i = 0)$ has a
(defective) density $f_i\dvtx \SS\times[0,\infty) \to[0,\infty)$, that is,
\[
\Pr_x(\tauopt\in dt, T_i = 0) = f_i(x,t)\,dt, \qquad   t \geq0.
\]
\end{lemma}

\begin{pf}
Suppose that $0 \leq x_1 \leq x_2 \leq x_3 \leq1$. Then the event $T_2
= 0$ is $\Pr_x$ null and consequently
$\Pr_x(\tauopt\in dt, T_2 = 0)$ vanishes for any $t$. That is,
 $f_2(x,t) = 0$.\

Existence of a density for $\Pr_x(\tauopt\in dt, T_i = 0)$, $i = 1,3,$
is essentially a corollary
of the decomposition of $\tauopt$ on $\{ T_i = 0 \}$ which was
discussed in the proof of Lemma
\ref{l:heurfnexists}. Let us consider the case $i = 1$ ($i = 3$ is
similar). Recall that if
$\exitST^{(i)}_a$ is the first hitting time of $a$ by $X_i$ and $x_1
\leq x_2 \leq x_3$ then
\[
\Pr_x (\tauopt\in B, T_1 = 0 ) = \Pr_x \bigl(\exitST
^{(2)}_1 + \exitST^{(3)}_1 \in B, \exitST^{(2)}_1 < \exitST
^{(2)}_{x_1}, \exitST^{(3)}_1 < \exitST^{(3)}_{x_1} \bigr).
\]
The right-hand side is the convolution of the sub-probability measures
\[
\Pr_x \bigl(\exitST^{(i)}_1 \in\cdot, \exitST^{(i)}_1 < \exitST
^{(i)}_{x_1} \bigr), \qquad  i = 1,2.
\]

Now, if $x_1 = x_2$, then $T_1 > 0$ almost surely under $\Pr_x$. Furthermore,
the assumptions $x_2 \leq x_3$ and $x \notin D$ imply $x_2 < 1$. So, we
may assume that $x_2$ is in
the interval $(x_1,1)$. In this case, $\{\exitST^{(2)}_1 < \exitST
^{(2)}_{x_1}\}$ is not null and $X_2$ can be conditioned,
via a Doob $h$-transform, to exit $(x_1,1)$ at the upper boundary. That is,
under the measure $\Pr_{x_2} (\cdot |\exitST^{(2)}_1 <
\exitST^{(2)}_{x_1}  )$, $X_2$ is a
regular diffusion on $(x_1,1]$ with generator\vspace*{1pt} $\Generator^h$ defined by
$\Generator^h f = (1/h)\Generator(hf)$, where
\[
h(x_2) \eqdef\Pr_{x_2} \bigl(\exitST^{(2)}_1 < \exitST
^{(2)}_{x_1} \bigr) = \frac{x_2 - x_1}{1-x_1}
\]
(e.g., Corollary~2.4, page~289 of \cite{MR1326606}) with absorption at
$1$. In particular,
% the conditioned process is a regular diffusion on $(x_1,1]$ with an
%%The conditioned process is again a diffusion
the law of the first hitting time,
$\Pr_x (\exitST^{(2)}_1 \in\cdot |\exitST^{(2)}_1 <
\exitST^{(2)}_{x_1}  )$, has a density
(page~154 of \cite{itomckean74}). Thus,
%, i.e., $\Pr_x (\exitST^{(2)}_1 \in\cdot |\exitST^{(2)}_1 <
%
\[
\Pr_x \bigl(\exitST^{(2)}_1 \in\cdot, \exitST^{(2)}_1 < \exitST
^{(2)}_{x_1} \bigr)
= \Pr_x \bigl(\exitST^{(2)}_1 \in\cdot |\exitST^{(2)}_1 <
\exitST^{(2)}_{x_1}  \bigr)\Pr_x \bigl(\exitST^{(2)}_1 <
\exitST^{(2)}_{x_1} \bigr)
% % &=&\Pr^h_{x_2} (\exitST_1 \in\cdot )h(x_2),
\]
is also absolutely continuous and $\Pr_x (\tauopt\in\cdot, T_1
= 0 )$ is the convolution of two measures, at least one of which
has a density.
\end{pf}

The next step is to show that $v$ solves \eqref{e:genddtvxt} in a
probabilistic sense.

\begin{lemma}\label{l:rep}
Fix $i \in V$ and define the function $u\dvtx \SS\times[0,\infty) \to\R
$ by
%
%e3.3 ###
\begin{equation}\label{e:rep}  \hspace*{25pt}
u(x,t) \eqdef\E_x\biggl [ v\bigl( X^{(i)}\bigl(t \wedge\rho^{(i)}\bigr), \bigl(t - \rho
^{(i)}\bigr)^+\bigr) - \int_0^{t \wedge\rho^{(i)}} f_i\bigl(X^{(i)}(s),
t-s\bigr)\,ds \biggr],
\end{equation}
where $\rho^{(i)} = \inf\{t \geq0\dvtx  X_i(t) \notin(0,1)\}$ and $f_i$
is the density from
Lemma \ref{l:densityexists}.
Then:
\begin{longlist}[(b)]
\item[(a)] for each $x \notin D$, $u(x,\cdot)$ has the same Laplace
transform as $v(x,\cdot)$,
\item[(b)] both $u(x,\cdot)$ and $v(x,\cdot)$ are right continuous,
and as a result,
\item[(c)] the tail probability $v$ is equal to $u$ and so has the
representation given in \eqref{e:rep}.
\end{longlist}
\end{lemma}

\begin{pf}
(a) The Laplace transform of the tail probability
%$v(x,t) = \Pr_x(\tauopt> t)$
is, for $x \notin D$,
\begin{eqnarray*}
\int_0^\infty v(x,t) \exp(-rt)\,dt & = & \E_x\biggl [ \int_0^{\infty
} \Indi{\tauopt> t} \exp(-r t)\,dt  \biggr] \\
& = & \E_x \biggl[ \int_0^{\tauopt} \exp(-r t)\,dt  \biggr] \\
& = & r^{-1}  \bigl( 1 - \hat v_r(x) \bigr),
\end{eqnarray*}
using Fubini's theorem to get the first equality (the integrand is
nonnegative). Furthermore, for $x \in D$, both $v(x,t)$ and $1 - \hat
v_r(x)$ vanish
and so in fact, for \textit{any} $x \in\SS$ we have
%
%e3.4 ###
\begin{equation}\label{e:laplacev}
\int_0^\infty v(x,t) \exp(-rt)\,dt = r^{-1}  \bigl( 1 - \hat
v_r(x) \bigr).
\end{equation}

Now, we consider the Laplace transform of $u$. By linearity of the
expectation operator,% and Laplace transform operations , of first use
%Fubini
%
\[
u(x,t) = \E_x \bigl[ v\bigl( X^{(i)}\bigl(t \wedge\rho^{(i)}\bigr), \bigl(t - \rho
^{(i)}\bigr)^+\bigr) \bigr] - \E_x \biggl[\int_0^{t \wedge\rho^{(i)}}
f_i\bigl(X^{(i)}(s), t-s\bigr)\,ds \biggr].
\]

%The linearity of the Laplace transform allows us to compute the
%Laplace transforms of the first and second members on the right hand
%side separately.
First, consider the Laplace transform of the first member of the
right-hand side:
\[
\int_0^\infty\E_x \bigl[ v\bigl( X^{(i)}\bigl(t \wedge\rho^{(i)}\bigr), \bigl(t - \rho
^{(i)}\bigr)^+\bigr) \bigr]\exp(-rt)\,dt.
\]

Applying Fubini's theorem, the preceding expression becomes
\[
\E_x \biggl[ \int_0^\infty v\bigl( X^{(i)}\bigl(t \wedge\rho^{(i)}\bigr), \bigl(t -
\rho^{(i)}\bigr)^+\bigr)\exp(-rt)\,dt \biggr],
\]
which can be decomposed into the sum % the integral into $(t <
\begin{eqnarray*}
&&\E_x \biggl[\int_0^{\rho^{(i)}} v\bigl( X^{(i)}(t), 0\bigr)\exp(-rt)\,dt
\biggr]\\
&& \qquad {}+ \E_x \biggl[\int_{\rho^{(i)}}^\infty v\bigl( X^{(i)}\bigl(\rho^{(i)}\bigr), t -
\rho^{(i)}\bigr)\exp(-rt)\,dt \biggr].
\end{eqnarray*}
%
% and
% \[
% \]

% and $(t \geq\rho^{(i)})$.
The first term in the sum is
%
%e3.5 ###
\begin{equation}\label{e:laplaceu1}
\E_x  \biggl[ \int_0^{\rho^{(i)}} v\bigl( X^{(i)}(t), 0\bigr)\exp(-rt)\,dt
 \biggr] = r^{-1} \E_x \bigl[1 - \exp\bigl(-r\rho^{(i)}\bigr) \bigr],
\end{equation}
because when $x \notin D$, $\Pr_x$-almost-surely we have $X^{(i)}(t)
\notin D$ for $t < \rho^{(i)}$.
% The second term, %integral over $t \geq\rho^{(i)}$ is
% \[
% \E_x \int_{\rho^{(i)}}^\infty v( X^{(i)}(\rho^{(i)}), t - \rho^{(i)})
% \]
As for the second term, we shift the variable of integration to $u = t
- \rho^{(i)}$ and then use \eqref{e:laplacev}
to show that it is equal to
%
%e3.6 ###
\begin{equation}\label{e:laplaceu2}
r^{-1} \E_x \bigl[ \exp\bigl(-r \rho^{(i)}\bigr)\bigl( 1- \hat v_r\bigl(X^{(i)}
\bigl(\rho
^{(i)}\bigr)\bigr) \bigr)  \bigr].
\end{equation}

The treatment of
%
%e3.7 ###
\begin{equation}\label{e:laplaceu3}
\int_0^\infty\E_x  \biggl[ \int_0^{t \wedge\rho^{(i)}}
f_i\bigl(X^{(i)}(s), t-s\bigr)\,ds  \biggr] \exp(-rt)\,dt
\end{equation}
proceeds in a similar fashion---exchange the expectation and
outer integral and then decompose the integrals into $t < \rho^{(i)}$
and $t \geq\rho^{(i)}$.
The integral over $t < \rho^{(i)}$ is
\[
\E_x  \biggl[ \int_0^{\rho^{(i)}}\!\! \int_0^{t} f_i\bigl(X^{(i)}(s), t-s\bigr)\,ds
\exp(-rt)\,dt  \biggr].
\]
Exchanging the integrals in $t$ and $s$ gives
\[
\E_x \biggl[ \int_0^{\rho^{(i)}} \!\!\int_s^{\rho^{(i)}}
f_i\bigl(X^{(i)}(s), t-s\bigr) \exp(-rt) \,dt\,ds  \biggr].
\]

For the integral over $t \geq\rho^{(i)}$, we again exchange the
integrals in $t$ and $s$ to give
\[
\E_x \biggl[ \int_0^{\rho^{(i)}} \!\!\int_{\rho^{(i)}}^{\infty}
f_i\bigl(X^{(i)}(s), t-s\bigr) \exp(-rt) \,dt\,ds  \biggr].
\]
Summing these final two expressions and substituting $u = t-s$ shows
that \eqref{e:laplaceu3}
is equal to
\[
\E_x \biggl[ \int_0^{\rho^{(i)}}\!\!  \int_{0}^{\infty} f_i\bigl(X^{(i)}(s),
u\bigr) \exp(-ru)\,du \exp(-rs)\,ds  \biggr].
\]

The Laplace transform is a linear operator, and so we may sum
\eqref{e:laplaceu1}--\eqref{e:laplaceu3} to
show that the
Laplace transform of $u$ is equal to
%
%e3.8 ###
\begin{eqnarray}\label{e:laplaceu4}
&&r^{-1}\E_x \bigl[1 - \exp\bigl(-r \rho^{(i)}\bigr) \hat v_r\bigl(X^{(i)}\bigl(\rho
^{(i)}\bigr)\bigr) \bigr]\nonumber
\\[-8pt]
\\[-8pt] && \qquad {}+ \E_x  \biggl[\int_0^{\rho^{(i)}} \hat
f_r^i\bigl(X^{(i)}(s)\bigr) \exp(-rs)\,ds \biggr],
\nonumber
\end{eqnarray}
where we have used
\[
\int_{0}^{\infty} f_i(x, u) \exp(-rt)\,du = \hat f_r^i(x)
\]
for $x \notin D$.

But, \eqref{e:laplaceu4} is exactly what we get by substituting the
representation for
$\hat v_r$ from Lemma \eqref{l:vrhatrep} into \eqref{e:laplacev}, and
so we are done.

(b) Right-continuity of $v$ in $t$ follows from the Monotone
Convergence theorem. A little more
work is required to see that $u$ is right-continuous. We begin by
observing that
if $\rho^{(i)} > t$ then $X_i$ has not been absorbed by time $t$ and
so, if $x \notin D$,
there is a $\Pr_x$-negligible set outside of which $X^{(i)}(t) \notin D$.

It follows that $\{X^{(i)}(t) \notin D, \rho^{(i)} > t \} = \{ \rho
^{(i)} > t \}$ up to a null set. Combining this with
the fact that $v(\cdot,0) = \Indi{\cdot\notin D}$ shows
\[
\E_x \bigl[ v\bigl(X^{(i)}\bigl(t \wedge\rho^{(i)}\bigr),
\bigl(t - \rho^{(i)}\bigr)^+\bigr)\Indi
{ \rho^{(i)} > t}  \bigr] = \Pr_x \bigl(\rho^{(i)} > t
\bigr)  \qquad \mbox{for } x \notin D.
\]
The latter is right-continuous in $t$ by the Monotone Convergence theorem.
The complementary expectation
\[
\E_x \bigl[ v\bigl(X^{(i)}\bigl(t \wedge\rho^{(i)}\bigr),
\bigl(t - \rho^{(i)}\bigr)^+\bigr)\Indi
{\rho^{(i)} \leq t}  \bigr]
\]
is equal to
\[
\E_x \bigl[ v\bigl(X^{(i)}\bigl(\rho^{(i)}\bigr)
,t - \rho^{(i)}\bigr)\Indi{\rho^{(i)}
\leq t}  \bigr],
\]
the right continuity of which follows from that of $v$ and the
indicator $\Indi{\rho^{(i)} \leq t}$, together with the Dominated
Convergence theorem.

We now consider the expectation of the integral,
\[
\E_x \biggl[ \int_0^{t \wedge\rho^{(i)}} f_i\bigl(X^{(i)}(s), t-s\bigr)\,ds
 \biggr].
\]
Using Fubini's theorem, we may exchange the integral and expectation to get
%
%e3.9 ###
\begin{equation}\label{e:rightctyxikilled}
\int_0^{t} \E_x\bigl [f_i\bigl(X^{(i)}(s), t-s\bigr)\Indi{\rho^{(i)} > s}
 \bigr]\,ds.
\end{equation}

This suggests the introduction of $(p^\dagger_s; s \geq0)$, the
transition kernel of $X_i$
killed (and sent to a cemetery state) on leaving $(0,1)$. Such a
density exists by the
arguments in Section~{4.11} of \cite{itomckean74}.

For notational ease, let us assume $i=1$, then \eqref
{e:rightctyxikilled} can be written
\[
\int_0^{t} \int_0^1 p^\dagger_s(x_1,y) f_1\bigl((y,x_2,x_3), t-s\bigr)\,dy\,ds.
\]
Finally, changing the variable of integration from $s$ to $s^\prime=
t-s$ gives
\[
\int_0^{t} \int_0^1 p^\dagger_{t - s^\prime}(x_1,y)
f_1\bigl((y,x_2,x_3), s^\prime\bigr)\,dy\,ds^\prime,
\]
and so regularity of \eqref{e:rightctyxikilled} in $t$ is inherited
from $p^\dagger$.
This is sufficient because $p_t^\dagger$ is continuous in $t > 0$
(again see \cite{itomckean74}).
%$\{X^{(i)}(t) \notin D, t

(c) It follows from (a) that for each $x \notin D$, $u(x,t)$ and
$v(x,t)$ are equal for almost every $t \geq0$. Hence, right continuity
is enough to show $v(x,t) = u(x,t)$ for every $t \geq0$.
\end{pf}

From the probabilistic representation for $v$, we need to deduce some
sub-martingale type inequalities
for $v(X^\straT(t), T-t)$, $0\leq t \leq T$. As we will see later,
it is enough to consider strategies that, for some $\epsilon> 0$, run
only one process
during the interval $(k\epsilon, (k+1)\epsilon)$, for integers $k
\geq0$. In other words, the rates for each process are either zero or
one and are constant over $(k\epsilon, (k+1)\epsilon)$.\vadjust{\goodbreak}%More specifically,
%see the following.%$\dot\straT_i^\epsilon$

\begin{definition}[($\epsilon$-strategy)]\label{d:epsilonstrat} For
$\epsilon> 0$ we let $\Pi_\epsilon$
denote the set of strategies $\straT^\epsilon$ such that for any
integer $k \geq0$,
\[
\straT^\epsilon(t) = \straT^\epsilon(k\epsilon) + (t - k\epsilon
)\xi_k, \qquad
k \epsilon\leq t \leq(k+1)\epsilon,
\]
where $\xi_k$ takes values in the set of standard basis elements $\{
(1,0,0), (0,1,0), \break(0, 0,1)\}$.
\end{definition}

\begin{lemma}\label{l:submgeachcmpnt}
Suppose $x \in\SS$ and $0 \leq t \leq T$, then the following
sub-martingale inequalities hold.
\begin{longlist}[(b)]
\item[(a)] For $i \in V$,
\[
\E_x  \bigl[ v\bigl(X^{(i)}(t), T-t\bigr) \bigr] \geq v(x,T).
\]
%
%i.e. $ ( v(X^{(i)}(t), T-t); t \leq T  )$ is an $\Filt_i$
%sub-martingale.
%
\item[(b)] If $\straT^\epsilon\in\Pi_\epsilon$ then
\[
\E_x \bigl[ v\bigl(X^{\straT^\epsilon}(t), T-t\bigr)  \bigr] \geq v(x,T).
\]
\end{longlist}
\end{lemma}

\begin{pf}
Consider first the quantity
%
%e3.10 ###
\begin{equation}\label{e:lsubmg1}
\E_x  \bigl[ \E_{X^{(i)}(t)} \bigl [ v\bigl( X^{(i)}\bigl((T-t) \wedge\rho
^{(i)}\bigr), \bigl(T- t - \rho^{(i)}\bigr)^+\bigr)  \bigr] \bigr].
\end{equation}

Our Markovian setup comes with a shift operator $\theta= \theta
^{(i)}$ for $X^{(i)}$
defined by $X^{(i)}\circ\theta_s(\omega,t) = X^{(i)}(\theta_s\omega
,t) = X^{(i)}(\omega,s+t)$ for each $\omega\in\Omega$.
Using the Markov property of $X^{(i)}$, \eqref{e:lsubmg1} becomes
\[
% \E_x \E_x  ( v( X^{(i)}(T-t \wedge\rho), (T- t - \rho)^+) \circ
\E_x \bigl [ \E_x  \bigl[ v\bigl( X^{(i)}\bigl( (T-t) \wedge\rho^{(i)}\bigr), \bigl(T-
t - \rho^{(i)}\bigr)^+\bigr)
 \circ\theta_t\big|\Filt_i(t)  \bigr]  \bigr].
\]

From here, use the Tower Property and the fact that $\rho^{(i)} \circ
\theta_t = (\rho^{(i)} - t) \vee0$
to find that \eqref{e:lsubmg1} equals
%
%e3.11 ###
\begin{equation}\label{e:lsubmgineq1}
\E_x  \bigl[ v\bigl( X^{(i)}\bigl(T \wedge\rho^{(i)}\bigr), \bigl(T - \rho
^{(i)}\bigr)^+\bigr) \bigr].
\end{equation}

We can give a similar treatment for
%
%e3.12 ###
\begin{equation}\label{e:lsubmg2}
\E_x  \biggl[ \E_{X^{(i)}(t)}  \biggl[ \int_0^{(T-t) \wedge\rho
^{(i)}} f_i\bigl(X^{(i)}(s), T-t-s\bigr)\,ds \biggr]  \biggr].
\end{equation}

Again using the Markov property of $X^{(i)}$, \eqref{e:lsubmg2} becomes
\[
\E_x  \biggl[ \E_x  \biggl[ \int_0^{(T-t) \wedge\rho^{(i)}}
f_i\bigl(X^{(i)}(s), T-t-s\bigr)\,ds \circ\theta_t\Big|\Filt_i(t) \biggr]  \biggr].
\]

Substituting in for $X^{(i)}\circ\theta_t$ and $\rho^{(i)} \circ
\theta_t$
and using the Tower Property, the latter expectation is seen to be
\[
\E_x \biggl [ \int_0^{(T-t) \wedge(\rho^{(i)}-t) \vee0}
f_i\bigl(X^{(i)}(s+t), T-t-s\bigr)\,ds \biggr].
\]

Now make the substitution $u = s + t$ in the integral and use the fact
that $f_i$ is nonnegative
to show that \eqref{e:lsubmg2} is less than or equal to
%
%e3.13 ###
\begin{equation}\label{e:lsubmgineq2}
\E_x  \biggl[ \int_0^{T \wedge\rho^{(i)}} f_i\bigl(X^{(i)}(u),
T-u\bigr)\,du \biggr].
\end{equation}

The final step is to note that, by Lemma \ref{l:rep},
\begin{eqnarray*}
v(x, T-t) &=& \E_x  \bigl[ v\bigl( X^{(i)}\bigl(T-t \wedge\rho^{(i)}\bigr),
\bigl (T- t -
\rho^{(i)}\bigr)^+\bigr) \bigr] \\
&& {}  - \E_x  \biggl[ \int_0^{(T-t) \wedge\rho^{(i)}} f\bigl(X^{(i)}(s),
T-t-s\bigr)\,ds \biggr],
\end{eqnarray*}
and so $\E_x[v(X^{(i)}(t), T-t)]$ is equal to \eqref{e:lsubmg1} minus
\eqref{e:lsubmg2},
which by the argument above is greater than or equal to%
\[
\E_x  \bigl[v\bigl( X^{(i)}\bigl(T \wedge\rho^{(i)}\bigr), \bigl(T - \rho
^{(i)}\bigr)^+\bigr) \bigr] - \E_x \biggl[\int_0^{T \wedge\rho^{(i)}}
f_i\bigl(X^{(i)}(u), T-u\bigr)\,du \biggr].
\]

Again appealing to Lemma \ref{l:rep} shows that the latter is exactly $v(x,T)$.

(b) It is sufficient to prove that for $k \epsilon\leq t \leq
(k+1)\epsilon$ we
have
%
%e3.14 ###
\begin{equation}\label{e:discretestratsubmg}
\E_x \bigl[ v\bigl(X^{\straT^\epsilon}(t), T-t\bigr) |\Filt^{\straT^\epsilon
}(k\epsilon) \bigr] \geq
v\bigl(X^{\straT^\epsilon}(k\epsilon), T-k\epsilon\bigr).
\end{equation}
The desired result then follows by applying the Tower Property of
conditional expectation and iterating this inequality.
If $X^{\straT^\epsilon}$ enjoys the Markov property, this inequality
follows from (a), but in general
our strategies can be non-Markov so we must do a little extra work.\vspace*{2pt}

Let us take $\nu\eqdef\straT^\epsilon(k\epsilon)$ and $\mcH\eqdef
\Filt^{\straT^\epsilon}(k\epsilon)$. Then $\nu$ takes values\vspace*{-1pt} in the
grid $\mcZ\eqdef\{0,\epsilon, 2\epsilon, \ldots \}^3$ and $\Lambda
\in\mcH$ implies
that $\Lambda\cap\{\nu= z\}$ is an element of the $\sigma$-field
$\Filt(z) = \sigma(\Filt_1(z_1),\ldots ,\Filt_3(z_3))$ for $z \in
\mcZ$. It follows
from the definition of conditional expectation that % if $Z$ is a $
$\Pr_x$-almost-surely we have
%
%e3.15 ###
\begin{equation}\label{e:condexpSP}
\E_x ( \cdot| \mcH ) = \E_x  ( \cdot| \Filt(z)
 ) \qquad  \mbox{on } \{\nu= z\}.
% \Indi{\nu= z}
\end{equation}

Now, suppose that $\xi_k \in\{ (1,0,0), (0,1,0), (0,0,1)\}$
defines the process that $\straT^\epsilon$ runs during the interval
$(k\epsilon, (k+1)\epsilon)$, that is,
\[
\straT^\epsilon(t) = \straT^\epsilon(k\epsilon) + (t - k\epsilon
)\xi_k, \qquad
k \epsilon< t < (k+1)\epsilon.
\]
By continuity of $\straT^\epsilon_i$ and right-continuity of $\Filt
^{\straT^\epsilon}$ (Lemma \ref{l:rightctyFc}),\vspace*{-1pt} $\xi_k$ must be
$\mcH$-measurable. So, if $A \eqdef A_1 \times A_2 \times A_3$ with
$A_i$ Borel measurable for each $i \in\V$,
\eqref{e:condexpSP} gives the equality
\[
\E_x \bigl(\Indi{\nu= z,  X^{\straT^{\epsilon}}(t) \in A,  \xi
_k = e_i}  | \mcH \bigr) =
\Indi{\nu= z,  \xi_k = e_i} \E_x \bigl(\Indi{X(z+(t-k\epsilon
)e_i) \in A}  | \Filt(z)  \bigr),
\]
where $X(z) = (X_1(z_1), X_2(z_2), X_3(z_3))$.

Next, we use the facts that $\Indi{X_j(z_j) \in A_j}$ is $\Filt(z)$
measurable for each $j$ and that the filtration $\Filt_i$ of $X_i$ is
independent of $\Filt_j$ for $j \neq i$, to show that the preceding
expression is equal to
\[
\Indi{\nu= z,  \xi_k = e_i, X_j(z_j) \in A_j, j \neq i} \E_x
\bigl[\Indi{X_i(z_i+(t-k\epsilon)) \in A_i}  | \Filt_i(z_i)  \bigr].
\]

Finally, the Markov property of $X_i$ allows us to write this as
\[
\Indi{\nu= z,\xi_k = e_i} \E_{X(z)} \bigl[\Indi
{X^{(i)}(t-k\epsilon) \in A} \bigr].
\]

As $\E_\cdot[v(X^{(i)}(t),s)]$ is Borel measurable for any $s, t \geq
0$, this is enough to conclude that in our original notation, on $\{\xi
_k = e_i\}$,
%
%e3.16 ###
\begin{equation}\label{e:discretestratFinalEq} \quad
\E_x \bigl[v\bigl(X^{\straT^{\epsilon}}(t),T-t\bigr) |\Filt^{\straT^\epsilon
}(k\epsilon) \bigr]
= \E_{X^{\straT^{\epsilon}}(k\epsilon)} \bigl[v\bigl(X^{(i)}(t-k\epsilon
),T-t\bigr)  \bigr].
\end{equation}

But part (a) shows that
\[
\E_{x}\bigl [v\bigl(X^{(i)}(t-k\epsilon),(T - k\epsilon) -(t - k\epsilon
)\bigr) \bigr] \geq v(x,T-k\epsilon),
\]
and so the right-hand side of \eqref{e:discretestratFinalEq} is
greater than or equal to
$v(X^{\straT^{\epsilon}}(k\epsilon), T-k\epsilon)$.
%\rightqed
\end{pf}

%s3.1 ###
\subsection{\texorpdfstring{Proof of Theorem
\protect\ref{t:stochminimality}}{Proof of Theorem 1.1}}

It is now relatively painless to combine the ingredients above. We take
an arbitrary strategy
$\straT$, use Lemma \ref{l:approxstrat} to approximate it by the
family $\straT^\epsilon$, $\epsilon> 0$, and then use Lemma \ref
{l:submgeachcmpnt} part (b) with $t = T \geq0$ to show that
\[
\Pr_x( \tau^{\straT^\epsilon} > T ) = \E_x [ v(X^{\straT
^\epsilon}(T), 0) ] \geq v(x,T)
\]
for any $x \notin D$ (equality holds trivially for $x \in D$).

The approximations are such that $\straT(t) \preceq\straT^\epsilon
(t + M\epsilon)$ for some constant $M > 0$. Thus, $\tau^\straT\leq
t$ implies that $\tau^{\straT^\epsilon} \leq t+ M\epsilon$.
More usefully, the contrapositive is that $\tau^{\straT^\epsilon} >
t + M \epsilon$ implies $\tau^\straT> t$ and so monotonicity of the
probability measure $\Pr_x$ then ensures
\[
\Pr_x( \tau^{\straT} > t ) \geq \Pr_x( \tau^{\straT^\epsilon} >
t + M\epsilon) \geq v(x,t + M\epsilon).
\]

Taking the limit $\epsilon\to0$ and using right continuity of
$v(x,t)$ in $t$
completes the proof.

%s4 ###
\section{Existence and almost sure uniqueness of $\straTopt$}\label
{s:stratexists}\label{sec4}

In this section, we give a proof for Lemma \ref{l:stratexists}. Recall
that we
wish to study strategies $\straT$ that satisfy the property
(RTM) for each $i \in V$, $\straT_i$ increases at time $t
\geq0$ [i.e., for every $s > t$, $\straT_i(s) > \straT_i(t)$]
% $\straT^\star_i$ increases only at times $t \geq0$ %(i.e. )
only if
\[
X_j^{\straT}(t) \leq X_i^{\straT}(t) \leq X_k^{\straT}(t)
\]
for some choice $\{j,k\} = V-\{i\}$.

Our idea is to reduce the existence and uniqueness of our strategy to a
one-sided
problem. Then, we can use the following result, taken from Proposition~5 and Corollary~13 in
\cite{mandelbaum1987cma} (alternatively Section~{5.1} of \cite
{KaspiMandelbaum95} or Section~{2} of \cite{barlow98variably}).

\begin{lemma}\label{l:mandelbaumstrat}
Suppose that $(Y_i(t); t \geq0)$, $i = 1,2,$ are independent and
identically distributed regular It{\^o} diffusions on $\R$, beginning at
the origin and with
complete, right continuous filtrations $(\FiltH_i(t); t \geq0)$. Then:
\begin{longlist}[(b)]
\item[(a)] There exists a strategy $\gamma= (\gamma_1(t), \gamma
_2(t); t \geq0)$
[with respect to the multiparameter filtration $\FiltH= (\sigma
(\FiltH_1(z_1),\FiltH_2(z_2)); z \in\R_+^2)$] such that $\gamma_i$
increases only at times $t \geq0$ with
\[
Y^\gamma_i(t) = Y^\gamma_1(t) \wedge Y^\gamma_2(t),
\]
that is, ``$\gamma$ follows the minimum of $Y_1$ and $Y_2$.''
\item[(b)] If $\gamma^\prime$ is another strategy with this
property, then, almost surely,
$\gamma^\prime(t) = \gamma(t)$ for every $t \geq0$. That is,
$\gamma$ is a.s. unique.
\item[(c)] The maximum $Y^\gamma_1(t) \vee Y^\gamma_2(t)$ increases
with $t$.
\end{longlist}
\end{lemma}

We first consider the question of uniqueness, it will then be obvious
how $\straTopt$ must
be defined. Suppose that $\straT$ is a strategy satisfying (RTM).

If $X_1(0) < X_2(0) = X_3(0)$, then $\straT$ cannot run $X_1$ (i.e.,
$\straT_1$ does not increase)
before the first time $\nu$ that either $X^\straT_2$ or $X^\straT_3$
hit $X_1(0)$.
Until then (or until a decision is made, whichever comes first),
$\straT_2$ may
increase only at times $t \geq0$ when $X^\straT_2(t) \leq X^\straT_3(t)$
and $\straT_3$ only when $X^\straT_3(t) \leq X^\straT_2(t)$. Hence,
on $\tau^\straT\wedge\nu\geq t$, the value of $\straT(t)$ is
determined by the strategy
in Lemma \ref{l:mandelbaumstrat}. Now, $X^\straT_2 \vee X^\straT_3$ increases
during this time, and so if $\nu< \tau^\straT$, we have
\[
X_1(0) = X^\straT_1(\nu) = X^\straT_2(\nu) \wedge X^\straT_3(\nu)
< X^\straT_2(\nu) \vee X^\straT_3(\nu).
\]
So again, we are in a position to apply the argument above,
and can do so repeatedly until a decision is made. In fact, it takes
only a finite number of iterations of the argument to determine $\straT(t)$
for each $t \geq0$ (on $\tau^\straT\geq t$) because each diffusion
$X_i$ is continuous,
the minimum $X^\straT_1 \wedge X^\straT_2 \wedge X^\straT_3$ is decreasing
and the maximum $X^\straT_1 \vee X^\straT_2 \vee X^\straT_3$ increasing.
If $X_1(0) < X_2(0) < X_3(0)$, then $\straT$ must run $X_2$
exclusively until it hits
either $X_1(0)$ or $X_3(0)$. From then on, the arguments of the
previous case apply.

The remaining possibility is that $X_1(0) = X_2(0) = X_3(0) = a \in(0,1)$.
We shall define random times $\nu_\epsilon$, $0 < \epsilon< (1 -
a)\wedge a$, such that:
\begin{itemize}
\item$\straT(\nu_\epsilon)$ is determined by the property (RTM),
\item under some labeling, either
\[
a - \epsilon < X_1^\straT(\nu_\epsilon) < a < X^\straT_2(\nu
_\epsilon) = X^\straT_3(\nu_\epsilon)= a+\epsilon
\]
or
\[
a - \epsilon = X_1^\straT(\nu_\epsilon) = X^\straT_2(\nu_\epsilon
) < a < X^\straT_3(\nu_\epsilon) < a+\epsilon
\]
and
\item$\nu_\epsilon\to0$ as $\epsilon\to0$.
\end{itemize}
Again, we may then use the one-sided argument to see that, almost
surely, on $\nu_\epsilon\leq t \leq\tau^\straT$, $\straT(t)$ is
determined by (RTM). This is sufficient because $\nu_\epsilon\to0$
as $\epsilon\to0$.

To construct $\nu_\epsilon$, suppose, without loss of generality,
that $X_1$ and $X_2$
both exit $(a -\epsilon, a+ \epsilon)$ at the upper boundary. We
denote by $\alpha_i$
the finite time taken for this to happen, that is, %by $X_i$ to exit the
%interval, i.e.
%
\[
\alpha_i \eqdef\inf\{ t > 0\dvtx  X_i(t) \notin(a - \epsilon, a +
\epsilon) \}.
\]

Define
\[
l_i \eqdef\inf_{0 \leq s \leq\alpha_i} X_i(s)
\]
to be the lowest value attained by $X_i$ before it exits
$(a - \epsilon, a + \epsilon)$. It follows from Proposition~5 of
\cite{mandelbaum1987cma} that it is
almost sure that the $l_i$ are not equal and so, we may assume that
$l_3 < l_2 < l_1$ (by relabeling if necessary).

Intuitively, (RTM) means that $X^\straT_1$ and $X^\straT_2$ should hit
$a+\epsilon$ together while $X^\straT_3$ gets left down at $l_2$.
We already know it takes time $\alpha_i$ for $X_i$ to hit $a+\epsilon
$ ($i = 1,2$)
and $X_3$ takes time
\[
\beta_3 \eqdef\inf\{ t > 0 \dvtx  X_3(t) = l_2 \}
\]
to reach $l_2$. So, we set $\nu_\epsilon= \alpha_1 + \alpha_2 +
\beta_3$,
and claim that
\[
\straT(\nu_\epsilon) = (\alpha_1,\alpha_2,\beta_3).
\]
The proof proceeds by examining the various cases. Firstly, if
$\straT_1(\nu_\epsilon) > \alpha_1$ and $\straT_2(\nu_\epsilon)
\geq\alpha_2$, then necessarily $\straT_3(\nu_\epsilon) < \beta_3$
and $X_3(z_3) > l_2$ for any $z_3 \leq\straT_3(\nu_\epsilon)$. But,
then there
exist times $\alpha^\prime_i < \straT_i(\nu_\epsilon)$ ($i=1,2$) with
\[
l_2 = X_2(\alpha^\prime_2) < X_3(z_3) < X_1(\alpha^\prime_1) = a +
\epsilon
\]
for any $z_3 \leq\straT_3(\nu_\epsilon)$, contradicting (RTM).

The second case is that $\straT_1(\nu_\epsilon) < \alpha_1$
and $\straT_2(\nu_\epsilon) \leq\alpha_2$. Necessarily, we then have
$\straT_3(\nu_\epsilon) > \beta_3$.
Now, $X_i(z_i) \geq l_2$ for $z_i \leq\alpha_i$, $i=1,2,$ and so
(RTM) implies that $X_3(z_3) \geq l_2$ as well for $z_3 \leq\straT
_3(\nu_\epsilon)$.
In addition, (RTM) and $\straT_3(\nu_\epsilon) > \beta_3$ imply that
\[
\straT_2(\nu_\epsilon) \geq\inf\{ t>0\dvtx  X_2(t) = l_2 \}
\]
[otherwise $X_3(\beta_3) < X_i(z_i)$ for $z_i \leq\straT_i(\nu
_\epsilon)$, $i=1,2$].
So, both $X_2$ and $X_3$ have attained $l_2$
and then stayed above it for a positive amount of time.
But, by Proposition~5 in \cite{mandelbaum1987cma}, this event (that
``the lower envelopes of $X_2$ and $X_3$ are simultaneously flat'') has
probability zero.

The final case $\straT_1(\nu_\epsilon) > \alpha_1$ and $\straT
_2(\nu_\epsilon) \leq\alpha_2$
has two subcases, $\straT_3(\nu_\epsilon) \leq\beta_3$ and $\straT
_3(\nu_\epsilon) > \beta_3$---both can be eliminated by the methods above. The only
remaining possibility is that $\straT_i(\nu_\epsilon) = \alpha_i$
for $i = 1,2$
and $\straT_3(\nu_\epsilon) = \beta_3$.

The discussion above tells us how to define $\straTopt$---if $X_1(0)
< X_2(0) \leq X_3(0)$
under some labelling, then we just alternate the one-sided construction
from Lemma
\ref{l:mandelbaumstrat} repeatedly to give a strategy satisfying (C1)--(C3).
If $X_1(0) = X_2(0) = X_3(0) = a \in(0,1)$, take $0< \epsilon< a
\wedge(1-a)$
and define $\straTopt(\nu_u)$, $0 < u \leq\epsilon$, via the
construction above.
Now, $\nu_u$ is only left continuous, so we have yet to define
$\straTopt$ on
the stochastic intervals $(\nu_u, \nu_{u+}]$, $u \leq\epsilon$.
But, this is easily done because
$X^\straTopt(\nu_u)$ has exactly two components equal and so we can
again use the
one-sided construction on this interval. We define $\straTopt$ on
$(\nu_\epsilon, \tau^\straTopt]$ similarly.
The properties (C1) and (C2) are readily verified. To confirm (C3), we
first note that $\straTopt$ satisfies (RTM). But (RTM)
gives us almost sure uniqueness of the paths of $\straTopt$. It
follows that
our definition of $\straTopt$ does not depend on $\epsilon$.
The second observation, which is not trivial, is that $\straT$
satisfies (C3)
with respect to the filtration $\Filt^\epsilon$ obtained
by\vspace*{-1.5pt} enlarging $\Filt$ to include $\bigvee_{i=1}^3 \Filt_i(\alpha
^\epsilon_i)$,
where $\alpha^\epsilon_i \eqdef\inf\{ t > 0\dvtx  X_i(t) \notin(a -
\epsilon, a + \epsilon) \}$.
That is, $\Filt^\epsilon$ contains the information necessary to
construct $C(\nu_\epsilon)$.
Property (C3) follows because $\Filt^\epsilon(\eta) \to\Filt(\eta
)$ as $\epsilon\to0$
for any $\eta\in\R_+^3$.

%s5 ###
\section{$X^\straTopt$ as a doubly perturbed diffusion}\label{s:perturbedBM}\label{sec5}

We now turn our attention to the optimally controlled process
$X^\straTopt$.
For convenience, we will work with the minimum
\[
I_t \eqdef X^\straTopt_1(t) \wedge X^\straTopt_2(t)\wedge X^\straTopt_3(t),
\]
maximum
\[
S_t \eqdef X^\straTopt_1(t) \vee X^\straTopt_2(t)\vee X^\straTopt_3(t)
\]
and middle value
\[
M_t \eqdef\bigl(X^\straTopt_1(t) \vee X^\straTopt_2(t)\bigr) \wedge
\bigl(X^\straTopt_1(t) \vee X^\straTopt_3(t)\bigr) \wedge\bigl(X^\straTopt_2(t)
\vee X^\straTopt_3(t)\bigr), \qquad  t \geq0,
\]
of the components of $X^\straTopt$ [so, if $X^\straTopt_1(t) \leq
X^\straTopt_2(t) \leq X^\straTopt_3(t)$, then $I_t = X^\straTopt
_1(t)$, $M_t = X^\straTopt_2(t), S_t = X^\straTopt_3(t)$].
There is no ambiguity when the values of the components are equal since
we are not formally
\textit{identifying} $I_t$, $M_t$ and $S_t$ with a particular component
of $X^\straTopt$.

Clearly, $M$ behaves as an It{\^o} diffusion solving \eqref{e:itosde} away from
the extrema $[0,1]$ and $S$, while at the extrema it experiences a perturbation.
This behavior is reminiscent of \textit{doubly perturbed Brownian motion},
which is defined as the (pathwise unique) solution $(X^\prime_t; t
\geq0)$
of the equation
\[\label{e:DPBM}
X^\prime_t = B^\prime_t + \alpha\sup_{s \leq t} X^\prime_s + \beta
\inf_{s \leq t} X^\prime_s,
\]
where $\alpha, \beta< 1$ and $(B^\prime_t; t \geq0)$ is a Brownian motion
starting from the origin. This process was introduced by Le Gall and Yor
in \cite{le1986excursions}; the reader may consult the survey \cite
{perman1997pbm}
and introduction of \cite{Chaumont2000219} for further details. In
Section~{2} of \cite{Chaumont2000219}, this definition is
generalized to accommodate nonzero initial values for the maximum
and minimum processes in the obvious way---if $i_0, s_0 \geq0$,
we take
\[
X^\prime_t = B^\prime_t + \alpha\Bigl ( \sup_{s \leq t} X^\prime_s
- s_0 \Bigr)^+ - \beta
 \Bigl( \inf_{s \leq t} X^\prime_s + i_0 \Bigr)^-,
\]
that is, $X^\prime$ hits $-i_0$ or $s_0$ before the perturbations begin.
As usual,
$a^+ = \max(a,0)$ and $a^- = \max(-a,0)$.

Our suspicion that $M$ should solve this
equation if the underlying processes are Brownian motions
is confirmed in the following lemma.

\begin{lemma}\label{l:MisaDPBM}
Suppose that $0 \leq i_0 \leq m_0 \leq s_0 \leq1$ and $\sigma= 1$.
Then, under $\Pr_{(i_0,m_0,s_0)}$,
there is a standard Brownian motion $(B^\prime_t; t \geq0)$ (adapted
to $\Filt^\straTopt$)
for which the process $M^\prime= M_t - m_0$, $t \geq0$, satisfies
\[
M^\prime_t = B^\prime_t -  \Bigl(\sup_{s \leq t} M^\prime_s -
s_0^\prime \Bigr)^+
+  \Bigl(\inf_{s \leq t} M^\prime_s + i_0^\prime \Bigr)^-, \qquad   t \leq
\tauopt,
\]
where $i_0^\prime= m_0 - i_0$ and $ s_0^\prime= s_0 - m_0$. In other words,
$M$ is a doubly perturbed Brownian motion with parameters $\alpha=
\beta= -1$.
\end{lemma}

\begin{pf}
For simplicity we can, and do, ignore the fact that the $X_i$ are
absorbed on leaving $(0,1)$
as $\straTopt$ does not run any absorbed process before the decision time.

The multiparameter martingale $( X_1(z_1) + X_2(z_2) + X_3(z_3); z \in
\R^3_+)$
is bounded and right continuous. Hence, Theorem \ref{t:MPCSM} implies that
\[
\xi_t \eqdef X^\straTopt_1(t) + X^\straTopt_2(t) + X^\straTopt
_3(t), \qquad  t \geq0,
\]
is a continuous (single parameter) martingale with respect to the
filtration $\Filt^\straTopt$. But, the $X_i$ are independent Brownian
motions and so the same argument applies to the multiparameter martingale
\[
\bigl ( \bigl(X_1(z_1) + X_2(z_2) + X_3(z_3)\bigr)^2 - (z_1 + z_2 + z_3); z \in
\R^3_+ \bigr),
\]
that is, $\xi^2_t - t$ is a martingale. It follows that $(\xi_t; t
\geq0)$ is a
Brownian motion with $\xi_0 = i_0 + m_0 +s_0$ and we can
take $B^\prime= \xi- (i_0 + m_0 +s_0)$.

Now, $\straTopt$ always ``runs $M$'' away from the extrema $[0,1]$ and
$S$ of
$X^\straTopt$ and so % it is no surprise that
\[
I_t = \inf_{s \leq t} M_s \wedge i_0,  \qquad   S_t = \sup_{s \leq t} M_s
\vee s_0,
\]
relationships which can be proved using the arguments of Section \ref
{s:stratexists}.
It follows that
\[
M^\prime_t = M_t - m_0 = \xi_t - m_0 - S_t - I_t = B^\prime_t
- \sup_{s \leq t} M_s \vee s_0+ s_0 - \inf_{s \leq t} M_s \wedge i_0
+ i_0.
\]

The result now follows by noting that for real $a$ and $b$
we have $a \wedge b - b = -(a - b)^-$
and $a \vee b - b = (a - b)^+$.
% The next step is to notice that $\xi_t = I_t + M_t + S_t$.
\end{pf}

Lemma \ref{l:MisaDPBM} is relevant because $\tauopt$ is precisely
the time taken for the doubly perturbed Brownian motion $M$ to exit
the interval $(0,1)$. In particular,
the expression we find for the Laplace transform
$\hat v_r(x)$ can be recovered from Theorems~4 and~5 in
Chaumont and Doney \cite{chaumontdoney00}.

We have so far assumed that $\sigma= 1$ and are yet to say anything
about more
general ``perturbed diffusion processes.'' %Lemma \ref{l:MisaDPBM}
%generalises easily to
There are several papers that
consider this problem. Doney and Zhang \cite{doneyzhang05} consider
the existence and
uniqueness of diffusions perturbed at their maximum. More recently,
Luo \cite{luo2009} has shown that solutions to %considers solutions
%to, in particular,
%
%e5.1 ###
\begin{equation}\label{e:DPDP}
X^\prime_t = \int_0^t \mu(s,X^\prime_s)\,ds + \int_0^t \sigma
(s,X^\prime_s)\,dB^\prime_s + \alpha\sup_{s \leq t} X^\prime_s +
\beta\inf_{s \leq t}X^\prime_s,
\end{equation}
exist and are unique, but only in the case that %the requirement that
$|\alpha| + |\beta| < 1$. A more general perturbed process is
considered in
\cite{lanyingyong2009} but similar restrictions on $\alpha$ and
$\beta$ apply.

%Lipshitz $\mu$ and $\sigma$ and $\alpha< 1$. As far as we know,
That is, there are no existence and uniqueness results for doubly
perturbed diffusions which cover our choice of $\alpha$ and $\beta$,
and less still for the Laplace transform of the distribution of
the time taken to exit an interval.

This is where our results seem to contribute something new. Lemma \ref
{l:MisaDPBM}
easily generalises to continuous $\sigma> 0$, and this combined with
the other results in this paper, lets us see that if %\eqref{e:DPDP}
%has a solution when $\alpha= \beta= -1$,
$\mu$ is bounded and Borel measurable and $\sigma> 0$ is continuous,
then there is a solution to
\[
M^\prime_t = \int_0^t \mu(M^\prime_s)\,dB^\prime_s + \int_0^t
\sigma(M^\prime_s)\,dB^\prime_s - \sup_{s \leq t} M_s - \inf_{s
\leq t}M_s.
\]
Furthermore, we can compute the Laplace transform of the distribution
of the time taken for any solution of
this equation to exit any interval $(-a,b)$ when $\mu$ is zero.

\begin{remark}
While this paper was in review, we became aware of \cite
{2010arXiv10035844B}, which contains
an existence result for \eqref{e:DPDP} covering $\alpha= \beta= -1$.
\end{remark}

%s6 ###
\section{Majority decisions of $2k+1$ diffusions and veto voting}\label{sec6}
The problem that we have solved has a natural generalization in which
there are $m$ diffusions
instead of three. In particular, one might ask for the majority
decision of an odd
number of ``diffusive voters'' $(X_i(t); t \geq0)$, $i = 1,\ldots , m$.
We believe that the optimal strategy is still to ``run the middle.'' In
other words,
if $m = 2k+1$, %then $\straT^\star_{j(k+1)}$ should increase only at
%times $t \geq0$ with
and
\[
%X^\straTopt_{j(i)}(t) \leq X^\straTopt_{j(i+1)}(t), 1 \leq i \leq n - 1
X^\straTopt_1(t) \leq\cdots  \leq X^\straTopt_k(t) < X^\straTopt_{k+1}(t)
< X^\straTopt_{k+2}(t) \leq\cdots  \leq X^\straTopt_m(t)
\]
%
% under some labelling $\{j(1), \ldots , j(n)\} = \{1,\ldots ,m\}$.
then $\straT^\star_{k+1}$ increases at unit rate
until $X^\straTopt_{k+1}$ hits either $X^\straTopt_k(t)$ or
$X^\straTopt_{k+2}(t)$.

Another variant of majority voting is ``veto voting,'' where
we have an arbitrary number $m^\prime> 0$ of diffusions,
and declare a negative decision if at least $k \leq m^\prime$ of them get
absorbed at the lower boundary (otherwise, no veto occurs and a
positive decision is made).
In fact, this is a special case of majority voting in which
some of the processes begin in an absorbed state. For example,
consider the case $2k < m^\prime$. This implies there is no veto if the
majority of voters return positive decisions. This is equivalent to
asking for a majority of $m = 2(m^\prime-k)+1$
diffusive voters, with $m+1-2k$ of them beginning in a state of
absorption at zero. The case $2k \geq m^\prime$ admits a similar description
in terms of majority voting. The analogue of the ``run the middle'' conjecture
is that if
\[
X^\straTopt_1(t) \leq\cdots  \leq X^\straTopt_{k-1}(t) < X^\straTopt_{k}(t)
< X^\straTopt_{k+1}(t) \leq\cdots \leq  X^\straTopt_{m^\prime}(t)
\]
then $\straT^\star_{k}$ should increase at unit rate
until $X^\straTopt_{k}$ hits either $X^\straTopt_{k-1}(t)$ or
$X^\straTopt_{k+1}(t)$.
In other words, we ``run the component with $k${th} order statistic.''
The extreme of this is true veto voting in which a single diffusion being
absorbed at zero will veto the others. This is the case $k=1$, and the
conjecture is that we should always ``run the minimum'' of the diffusions.

In principle, this conjecture could be tackled using the methods of
this paper
since the heuristic argument used to compute the Laplace transform of
the distribution
of the decision time still applies. The difficulty arises
because we cannot prove a more general existence result for solutions to
the analogue of \eqref{e:heureq4}.

One might also consider diffusions which obey different stochastic
differential equations. We have found an implicit equation for the
switching boundaries in the optimal strategy for $m^\prime= 2, k = 1$
``veto voting''
problem by solving a free boundary problem. However, we have no
conjecture for the general solution.

\begin{appendix}
\section*{Appendix: Results for multiparameter processes}\label{s:intromultiparam}
The proofs of Lemmas \ref{l:vhatverification} and \ref{l:MisaDPBM} appealed
to the fact that a %the multiparameter time change of a
multiparameter martingale composed with a strategy is again a martingale.
Moreover, it was asserted that we can approximate
an arbitrary strategy with a discrete one. This appendix contains a precise
statement of these results, together with basic definitions (adopted
from Section~{4} of
\cite{KarouiKaratzas97}).

Let $(\Omega, \Filt, \Pr)$ be a complete probability space, $\R_+$
denote the set
of nonnegative reals $[0,\infty)$ and $d \geq2$.
A family $ (\Filt(\eta),  \eta\in\R_+^d )$ of $\sigma
$-algebras contained in $\Filt$
is called a multiparameter filtration if, for every $\eta, \nu\in
\R_+^d$ with
$\eta\preceq\nu$,
\[
\Filt(\eta) \subseteq\Filt(\nu).
\]

We make the strong assumption that $\Filt$ is generated from independent
filtrations, as is in Section \ref{s:probstatement}; that is,
\[
\Filt(\eta) = \sigma(\Filt_1(\eta_1), \ldots , \Filt_d(\eta
_d)), \qquad   \eta\in\R_+^d,
\]
where $ (\Filt_i(t), t \geq0 )$, $i = 1,2,\ldots ,d,$ are
independent, right continuous, complete
filtrations. Note that this filtration satisfies the ``usual
conditions'' of
\cite{KarouiKaratzas97}.

A real-valued process $(Z(\eta),  \eta\in\R_+^d)$ is called a
multiparameter super-martingale with respect to $(\Filt(\eta),  \eta
\in\R^d)$
if for every $\eta$:
\begin{itemize}
\item$\E [|Z(\eta)| ] < \infty$, that is, $Z$ is integrable,
\item$Z(\eta)$ is $\Filt(\eta)$ measurable and
\item$\E [Z(\eta)|\Filt(\nu) ] \leq Z(\nu)$ for every
$\eta\preceq\nu$.
\end{itemize}

A strategy $\straT$ is a $\R_+^d$ valued process
such that $\straT_i$ increases from the origin, $\sum_i \straT_i(t)
= t$
and $\{\straT(t) \preceq\eta\} \in\Filt(\eta)$ for every $t \geq
0$ and
$\eta\in\R_+^d$ [conditions (C1)--(C3) from Section \ref{s:probstatement}].
For each strategy, we define a filtration $(\Filt^\straT(t),  t \geq
0)$ by
\[
\Filt^\straT(t) \eqdef\bigl\{ F \in\Filt\dvtx  F \cap\{\straT(t) \preceq
\eta\}
\in\Filt(\eta)  \ \forall\eta\in\R_+^d \bigr\},  \qquad   t \geq0.
\]

\begin{lemma}\label{l:rightctyFc}
$\Filt^\straT$ is right continuous.
\end{lemma}

\begin{pf}
Fix $t \geq0$ and suppose that $F \in\Filt^\straT(s)$ for every $s
> t$. We need
to show that $F \in\Filt^\straT(t)$, that is,
\[
F \cap \{ \straT(t) \preceq\nu \} \in\Filt(\nu
) \qquad  \mbox{for all } \nu\in\R_+^d.
\]

The trick is, for each $\nu\in\R_+^d$, to take a decreasing sequence
$\nu^n \in\R_+^d$, $n > 0$,
such that $\nu^n \to\nu$, $\nu^n_i > \nu_i$ and use continuity of
$\straT$ to write
\[
F \cap\{\straT(t) \preceq\eta\} = \bigcap_{m >0} \bigcup_{n > 0}
\{\straT(t + 1/n) \preceq\nu^m \} \cap F.
\]

By assumption, $F \in\Filt^\straT(t + 1/n)$ for each $n > 0$ and so,
by definition,
\[
\{\straT(t + 1/n) \preceq\nu^m \} \cap F \in\Filt(\nu^m)
\]
for each $m > 0$. Thus, the union
\[
A_m \eqdef\bigcup_{n > 0} \{\straT(t + 1/n) \preceq\nu^m \} \cap F
\]
is also in $\Filt(\nu^m)$. Because $\straT$ is increasing, we have
$A_{m+1} \subseteq A_m$
and so $\bigcap_{m >0} A_m = \bigcap_{m > k} A_m$ for any $k > 0$.
Hence, for any $k$,
\[
F \cap \{ \straT(t) \prec\nu \} = \bigcap_{m > k} A_m
\in\Filt(\nu^k).
\]

But, since $\Filt$ is generated from independent filtrations,
\[
\bigcap_k \Filt(\nu^k) = \Filt(\nu)
\]
by Lemma~2 of \cite{MR841588}.\footnote{A remark in this paper warns
that the conclusion may be false if the filtrations are not
independent!} This concludes the proof.
\end{pf}

The process
\[
Z^\straT\eqdef \bigl( Z_1(\straT_1(t)), \ldots , Z_d(\straT
_d(t));  t \geq0 \bigr)
\]
is adapted to this filtration. The idea is that $Z^\straT$ should be a
super-martingale with
respect to $\Filt^\straT$. Indeed, Proposition~4.3 in \cite
{KarouiKaratzas97} is
the following.
\begin{theorem}\label{t:MPCSM}
Suppose that $Z$ is a right continuous multi-parameter
super-martingale and that $\straT$ is a strategy.
%(both with respect to the filtration $(\Filt(\eta); \eta\in\R^d)$).
Then $Z^\straT$ is a (local) $\Filt^\straT$-super-martingale.
\end{theorem}

This theorem appears in various guises throughout the literature
(a good reference for the discrete case is Chapter~1
of \cite{CairoliDalang96}), we do not give the proof.
Merely, we will mention one of its stepping stones---approximation of an arbitrary strategy with a discrete one.

Recall from Definition \ref{d:epsilonstrat} that for any
$\epsilon> 0$, $\Pi_\epsilon$ denotes the set of
strategies which only increase in one component
over each interval $[k\epsilon, (k+1)\epsilon)$, $k = 0,1,\ldots ,$
that is, $\straT^\epsilon$ is in $\Pi_\epsilon$ if
$\dot\straT_i$ a.e. takes only values $0$ or $1$
and is constant on each interval $(k\epsilon, (k+1)\epsilon)$. The promised
approximation result is
the following lemma.

\begin{lemma}\label{l:approxstrat}
\textup{(a)} For any strategy $\straT$, there exist
a family of strategies $\straT^\epsilon\in\Pi_\epsilon$, $\epsilon
> 0$
that converge to $\straT$ in the sense that
\[
\lim_{\epsilon\to0} \sup_{t \geq0} |\straT(t) - \straT^\epsilon
(t)| = 0,
\]
where $|\cdot|$ is any norm on $\R^d$.

\textup{(b)} Moreover, there is a positive constant $M > 0$ for which $\straT
(t) \preceq\straT^\epsilon(t + M\epsilon)$ for every $t \geq0$.
\end{lemma}

%The existence and uniform convergence
Part (a) of this lemma is exactly Theorem~7 of Mandelbaum \cite
{mandelbaum1987cma} and part (b)
%is a corollary to the author's constructive proof.
follows from directly from the constructive proof of (a). The details
are omitted.
\end{appendix}

% imsref loaded by smiklovaite, 2010-12-10 09:56:22
%

\printaddresses

\end{document}